\documentclass[a4paper,12pt]{amsart}
\usepackage[left=2cm,top=2.5cm,right=2cm,bottom=2.5cm]{geometry}
\usepackage{amsmath}
\usepackage{amssymb}
\usepackage{amsfonts}
\usepackage{amsthm}
\usepackage{stmaryrd}
\usepackage{url}
\usepackage{hyperref}
\usepackage{fancyhdr}
\usepackage{mathrsfs}
\usepackage{mathtools}
\usepackage[demo]{graphicx}
\usepackage{color}
\usepackage{enumerate,mdwlist}

\usepackage{amstext}
\usepackage{array}

\usepackage{tikz}
\usetikzlibrary{matrix,positioning,arrows,decorations.pathmorphing}
\usetikzlibrary{tikzmark}

\usepackage[shortlabels]{enumitem}
\setlist[enumerate]{leftmargin=25pt}
\setlist[itemize]{leftmargin=25pt}

\newtheorem{thm}{Theorem}[section]
\newtheorem{lemma}[thm]{Lemma}
\newtheorem{prop}[thm]{Proposition}
\newtheorem{cor}[thm]{Corollary}
\newtheorem{conj}[thm]{Conjecture}

\theoremstyle{definition}
\newtheorem*{ex}{Example}
\newtheorem*{rem}{Remark}

\theoremstyle{definition}
\newtheorem{defn}[thm]{Definition}

\DeclareMathOperator{\ord}{ord}

\newcommand{\supp}{\ensuremath{\operatorname{supp}}}

\newcommand{\deficit}{\ensuremath{\operatorname{def}}}

\newcommand{\bra}[1]{{\left({#1}\right)}}

\newcommand{\ceil}[1]{{\lceil{#1}\rceil}}

\newcommand{\scal}[1]{{\left\langle{#1}\right\rangle}}
\newcommand{\set}[1]{{\left\{{#1}\right\}}}

\DeclarePairedDelimiter\abs{\lvert}{\rvert}

\newcommand{\al}{\ensuremath{\alpha}}
\newcommand{\be}{\ensuremath{\beta}}

\newcommand{\de}{\ensuremath{\delta}}

\newcommand{\ze}{\ensuremath{\zeta}}

\newcommand{\La}{\ensuremath{\Lambda}}
\newcommand{\la}{\ensuremath{\lambda}}
\newcommand{\sig}{\ensuremath{\sigma}}

\newcommand{\Om}{\ensuremath{\Omega}}

\newcommand{\ZZ}{\ensuremath{\mathbb{Z}}}

\newcommand{\znp}{\ensuremath{\ZZ_N^{\star}}}
\newcommand{\FF}{\ensuremath{\mathbb{F}}}
\newcommand{\QQ}{\ensuremath{\mathbb{Q}}}
\newcommand{\RR}{\ensuremath{\mathbb{R}}}
\newcommand{\NN}{\ensuremath{\mathbb{N}}}

\newcommand{\wt}[1]{\ensuremath{\widetilde{#1}}}
\newcommand{\wh}[1]{\ensuremath{\widehat{#1}}}

\newcommand{\CC}{\ensuremath{\mathbb{C}}}

\newcommand{\Gal}{\ensuremath{\textnormal{Gal}}}

\newcommand{\Php}{\ensuremath{\Phi_p(X^{N/p})}}
\newcommand{\Phq}{\ensuremath{\Phi_q(X^{N/q})}}

\newcommand{\sm}{\ensuremath{\setminus}}
\newcommand{\ssq}{\ensuremath{\subseteq}}

\newcommand{\Rar}{\ensuremath{\Rightarrow}}

\newcommand{\Lrar}{\ensuremath{\Leftrightarrow}}
\newcommand{\nssq}{\ensuremath{\nsubseteq}}
\newcommand{\vn}{\ensuremath{\varnothing}}

\newcommand{\ol}{\ensuremath{\overline}}

\newcommand{\pdiv}{\ensuremath{\mid\!\mid}}

\newcommand{\bs}[1]{\ensuremath{\mathbf{#1}}}

\newcommand{\Spq}{\ensuremath{\mathscr{S}_{p,q}}}
\newcommand{\Spqmin}{\ensuremath{\mathscr{S}'_{p,q}}}
\newcommand{\F}{\ensuremath{\mathscr{F}}}
\newcommand{\Fmax}{\ensuremath{\mathscr{F}_{\max}}}
\newcommand{\out}[1]{}

\definecolor{redi}{RGB}{255,38,0}
\definecolor{redii}{RGB}{200,50,30}
\definecolor{yellowi}{RGB}{255,251,0}
\definecolor{bluei}{RGB}{0,150,255}
\definecolor{blueii}{RGB}{135,247,210}
\definecolor{blueiii}{RGB}{91,205,250}
\definecolor{blueiv}{RGB}{115,244,253}
\definecolor{bluev}{RGB}{1,58,215}
\definecolor{orangei}{RGB}{240,143,50}
\definecolor{yellowii}{RGB}{222,247,100}
\definecolor{greeni}{RGB}{166,247,166}

\tikzset{ 
table/.style={
  matrix of nodes,
  row sep=-\pgflinewidth,
  column sep=-\pgflinewidth,
  nodes={rectangle,draw=black,text width=1.25ex,align=center},
  text depth=0.25ex,
  text height=1ex,
  nodes in empty cells
  },
texto/.style={font=\footnotesize\sffamily},
title/.style={font=\small\sffamily}
}

\numberwithin{equation}{section}

\title[On the structure of spectral and tiling subsets of cyclic groups]{On the structure of spectral\\ and tiling subsets of cyclic groups}



\author{Romanos Diogenes Malikiosis} 
\address{Aristotle University of Thessaloniki, Department of Mathematics, 541 24 Thessaloniki, Greece}
\email{romanos@math.auth.gr}

\begin{document}
\begin{abstract}
 The purpose of this paper is to investigate the properties of spectral and tiling subsets of cyclic groups, with an eye towards the spectral set conjecture
\cite{Fuglede} in one dimension, which states that a bounded measurable 
 subset of $\RR$ accepts an orthogonal basis of exponentials if and only if it tiles $\RR$ by translations. This conjecture is strongly connected to its discrete counterpart, namely that
 in every finite cyclic group, a subset is spectral if and only if it is a tile. The tools presented herein are refinements of recent ones used in the setting of cyclic groups; the structure of vanishing
 sums of roots of unity \cite{LL} is a prevalent notion throughout the text, as well as the structure of tiling subsets of integers \cite{CM}. We manage to prove the conjecture for cyclic groups of order
 $p^mq^n$, when one of the exponents is $\leq6$ or when $p^{m-2}<q^4$, and also prove that a tiling subset of a cyclic group of order $p_1^mp_2\dotsm p_n$ is spectral.
\end{abstract}
\maketitle

\bigskip
\bigskip
\section{Introduction}
\bigskip\bigskip

A basic fact in Fourier analysis is the decomposition of a periodic function into a series of orthogonal exponential functions. A natural question that arises is whether this concept can be generalized for functions
defined not necessarily on intervals $[0,T]$, where $T$ is the period, but on more general domains with some weaker conditions (for example, measurability or boundedness). This is the definition of a 
\emph{spectral set}. This, of course, can be generalized to higher dimensions as well, where periodic functions have a similar decomposition.

\begin{defn}
Let $\mu$ denote the Lebesgue measure on $\RR$.
 A bounded, measurable set $\Om\ssq\RR^d$ with $\mu(\Om)>0$ is called \emph{spectral}, if there is a discrete set $\La\ssq\RR^d$ such that the set of exponential functions $\set{e_{\la}(x)}_{\la\in\La}$,
 where $e_{\la}(x)=e^{2\pi i \la\cdot x}$, is a complete orthogonal set, that is
 \[\scal{e_{\la},e_{\la'}}_{\Om}=\int_{\Om}e_{\la-\la'}(x)dx=\de_{\la\la'}\mu(\Om),\]
 and every $f\in L^2(\Om)$ can be expressed as a Fourier series,
 \[f(x)=\sum_{\la\in\La}a_{\la}e_{\la}(x),\]
 for some $a_{\la}\in\CC$ (\emph{Fourier coefficients}).
\end{defn}

 Usually, we restrict periodic functions on a fundamental domain with respect to the period lattice; so $[0,1]^d$ is the
standard case, when the period lattice is $\ZZ^d$. In general, functions defined on fundamental parallelepipeds of a lattice $\La\ssq\RR^d$ have the same property; the exponentials then correspond to
$e_{\la}(x)$, where $\la$ ranges through the dual lattice $\La^{\star}$. We note that fundamental parallelepipeds are \emph{tiles} of $\RR^d$ with respect to the period lattice, say $\La$.

\begin{defn}
 A subset $A\ssq\RR^d$ \emph{tiles} $\RR^d$ by translations, if there is a translation set $T\ssq\RR^d$ such that almost all elements of $\RR^d$ have a unique representation as a sum $a+t$, where
 $a\in A$, $t\in T$. We will denote this by $A\oplus T=\RR^d$. $T$ is called the \emph{tiling complement} of $A$, and $(A,T)$ is called a \emph{tiling pair}.
\end{defn}

Given the premise in the standard periodic case, i.e. functions defined on $[0,1]^d$, it might be tempting to generalize this situation.
This was first attempted in 1974, when Fuglede \cite{Fuglede} stated the following conjecture, connecting spectral subsets and tiles of $\RR^d$.

\begin{conj}
 Let $\Om\ssq\RR^d$ be a bounded measurable set. $\Om$ accepts a complete orthonormal basis of exponentials if and only if $\Om$ tiles $\RR^d$ by translations.
\end{conj}

Fuglede proved this conjecture when either a spectrum or a tiling complement of $\Om$ is a lattice. All results proven in the next thirty years were in the positive direction, until Tao \cite{TaoFuglede}
disproved this conjecture, by constructing spectral subsets of $\RR^5$ that do not tile the space. Shortly thereafter, tiles with no spectra in $\RR^5$ had been found by Kolountzakis and Matolcsi \cite{KM06}.
The current state of art on Euclidean spaces is the failure of Fuglede's conjecture on dimensions 3 and above, in both directions \cite{FMM06,FarkasRevesz,KMhadamard,Mat}.
The noteworthy aspect of the construction of counterexamples is the passage to the setting of finite Abelian groups, which is also the setting of the present article. 
To be more precise, an example of a spectral subset of $\ZZ_8^3$ that does not tile is lifted to a counterexample in $\RR^3$ \cite{KMhadamard}, and similarly, an example of a tile with no spectra in $\ZZ_{24}^3$ is
lifted to a counterexample in $\RR^3$ \cite{FMM06}. The conjecture is still open in $\RR$ and $\RR^2$. In the one-dimensional case, this is best summarized as follows (\cite{DL14} and the
references therein):
\[\textbf{T-S}(\RR)\Longleftrightarrow\textbf{T-S}(\ZZ)\Longleftrightarrow\forall N\textbf{T-S}(\ZZ_N)\]
and
\[\textbf{S-T}(\RR)\Longrightarrow\textbf{S-T}(\ZZ)\Longrightarrow\forall N\textbf{S-T}(\ZZ_N),\]
where $\textbf{S-T}(G)$ denotes the fact that the Spectral$\Rar$Tiling direction holds in $G$, while $\textbf{T-S}(G)$ denotes that fact that the reverse implication is true.

We emphasize that the equivalence of the Spectral$\Rar$Tiling direction between the ``continuous'' and the ``discrete'' settings 
depends on \emph{rationality of spectrum}: the claim that every spectral subset of $\RR$ of measure $1$ has a spectrum in $\QQ$, which is a difficult problem in its own right. Another important
ingredient for this yet unproved equivalence is the periodicity of spectrum \cite{IK13}.

This conjecture has a renewed interest in the setting of finite Abelian groups during the last
few years, mostly due to the counterexamples mentioned above. In the cyclic case, \L{}aba first connected \cite{Laba} the work of Coven \& Meyerowitz \cite{CM}
on tiling subsets of integers to Fuglede's
conjecture, and proved subsequently that tiles of $\ZZ$ of cardinality $p^m$ or $p^mq^n$, where $p,q$ primes, are also spectral; the reverse direction holds when the cardinality is $p^m$. 
These arguments can be adapted to the finite cyclic group setting, proving that the Tiling$\Rar$Spectral direction is true for finite cyclic groups of order $p^mq^n$; for a self-contained account, 
we refer the reader to \cite{MK}. Moreover, both directions of Fuglede's conjecture hold for cyclic groups of order $p^m$; again, self-contained
proofs can be found in \cite{Qpfuglede} and \cite{MK}.

We mention in passing two another recent breakthroughs in Fuglede's conjecture; this conjecture has also been confirmed for the field of $p$-adic rational numbers $\QQ_p$
\cite{Qpfuglede2,Qpfuglede}, and for convex bodies in $\RR^d$ \cite{LM19}. We return to the setting of this paper, namely finite Abelian groups: for finite Abelian groups with two generators, 
it has been shown that Fuglede's conjecture holds 
in $\ZZ_p\times \ZZ_p$ \cite{IMP15}, a result later extended to $\ZZ_p\times\ZZ_{p^2}$ \cite{Shi19b} and very recently to $\ZZ_p\times\ZZ_{p^n}$ \cite{Zhang21}. When the generators are at least four,
the Spectral$\Rar$Tiling direction fails when the cardinality of the group is odd \cite{FS2020}. 

The goal of the present article is to develop tools and prove facts about spectral subsets of $\ZZ_N$, especially when $N$ has at most two distinct prime factors. We note that the conjecture is known to be
true when $N=p^nq$ \cite{MK} or $N=p^nq^2$ \cite{KMSV20}, i.e. when one of the exponents is $\leq2$. The methods and techniques developed herein extend these results, thus
``raising'' the exponent of $q$:


\begin{thm}\label{mainthm}
 Let $A\ssq\ZZ_N$ be a spectral set, where $N=p^mq^n$, with $p,q$ distinct primes. Then $A$ tiles $\ZZ_N$ if one of the following holds:
 \begin{enumerate}
  \item $p<q$ and $m\leq9$ or $n\leq6$.
  \item $p^{m-2}<q^4$.
 \end{enumerate}
\end{thm}

Fuglede's conjecture in $\ZZ_N$ when $N=p^mq^n$ seems within reach. It is the hope of the author that the methods developed here to ``raise'' the exponent, combined with the recent proofs of Fuglede's conjecture
in the groups $\ZZ_{pqr}$ \cite{Shi19a}, $\ZZ_{p^2qr}$ \cite{Som19}, and $\ZZ_{pqrs}$ \cite{KMSV21}, where the number of primes dividing $N$ is increased, will provide the groundwork towards a more effective attack on the one dimensional 
Fuglede's conjecture. 

Regarding the Tiling$\Rar$Spectral direction, we will provide a proof in cyclic groups of order $p_1^np_2\dotsb p_k$; this direction was previously known for groups of order $p^mq^n$ \cite{Laba}
(see also \cite{MK}), or of square-free order (proven in Terence Tao's blog\footnote{https://terrytao.wordpress.com/2011/11/19/some-notes-on-the-coven-meyerowitz-conjecture/} by \L{}aba and Meyerowitz; 
see also \cite{Shi19a}). Very recently, it was also proved for groups of order $(pqr)^2$ \cite{LL21a,LL21b}.

\begin{thm}\label{mainthm2}
 Let $N=p_1^np_2\dotsb p_k$, where $p_1,\dotsc,p_k$ are distinct primes. If $A\ssq\ZZ_N$ tiles, then $A$ is spectral.
\end{thm}

Even if Fuglede's conjecture is proven for $\ZZ_{p^mq^n}$, one should be cautious before conjecturing further that this might be true in all cyclic groups. For example, several strengthened
versions of the Tiling$\Rar$Spectral directions, such as the conjectures of Sands \cite{Sands} and Tijdeman \cite{Tijdeman}, are known to hold when $N$ has the
form $p^mq^n$, but they break down when $N$ has at least $3$ distinct prime factors, due to counterexamples given by Szabo \cite{Szabo}. 
Furthermore, the main tools used here, such as the structure of vanishing sums of $N$th
roots of unity \cite{LL} and primitive subsets of $\ZZ_N$, are much stronger when $N$ has at most two distinct prime factors. 



The different tools needed for the proofs of the above theorems are laid out in Sections 2 and 3. In Section 4 we prove \ref{mainthm2}, which is based on an inductive approach, and is relatively easy to follow.
The remaining sections are exclusively devoted to the two-prime case, $N=p^mq^n$, and they are much more technical and demanding. Theorem \ref{mainthm} is eventually proven in Section 11.

\bigskip\bigskip
\section{Multisets and mask polynomials}
\bigskip\bigskip

We denote by $\ZZ_N$ the cyclic additive group of $N$ elements. A multiset $A$ in $\ZZ_N$ is a collection of elements of $\ZZ_N$ along with some multiplicities; any element $a\in A$ appears $m_a\in\NN$ times
in $A$. The characteristic function of $A$ is denoted by $\bs 1_A$, and is defined by $\bs 1_A(a)=m_a$. It is understood that if $A$ is a proper set, this function takes only the values $0$ and $1$. The
support of the multiset $A$, denoted by $\supp A$, is simply the subset of $\ZZ_N$ of elements that appear at least once in $A$; in other words, $\supp A=\supp(\bs 1_A)$.

\begin{defn}
 Let $A$ be a multiset of elements in $\ZZ_N$. The \emph{mask polynomial} of $A$ is defined by $A(X)=\sum_{a\in A}m_a X^a$, and is considered as an element in $\ZZ[X]/(X^N-1)$. As such, the values of $A(X)$
 on the $N$th roots of unity $\ze_N^d$, $1\leq d\leq N$, (here $\ze_N=e^{2\pi i/N}$) are well-defined, i.e. independent of the representative of their class $\bmod(X^N-1)$. Also, for every $d\in\ZZ$ define the multiset
 $d\cdot A$ whose characteristic function is 
 \[\bs 1_{d\cdot A}(x)=\sum_{a\in\ZZ_N, da=x}\bs 1_A(a).\]
 The support of $d\cdot A$ will simply be denoted by $dA$. Finally, the multiset of elements of $A$ that are congruent to $j\bmod m$ for $m\mid N$ will be denoted as $A_{j\bmod m}$.
\end{defn}

\begin{ex}
 If $A=\set{0,3}\ssq\ZZ_6$,
then $2A=\set{0}$, while $2\cdot A=\set{0,0}$; moreover, $A(X)\equiv1+X^3\bmod(X^6-1)$, $(2A)(X)\equiv 1$, $(2\cdot A)(X)\equiv 2$. Also, $A_{0\bmod3}=A$, while $A_{1\bmod3}=A_{2\bmod3}=\vn$.
\end{ex}

In general, the following holds for the mask polynomial of $m\cdot A$:

\begin{prop}\label{maskpoly}
 Let $A$ be a multiset with elements from $\ZZ_N$. Then $A(X^m)\equiv(m\cdot A)(X)\bmod(X^N-1)$ for every $m\in\NN$. If $m\mid N$, then
 \[A(X^m)\equiv\sum_{j=0}^{N/m-1}\abs{A_{j\bmod N/m}}X^{jm}\bmod(X^N-1).\]
\end{prop}

\begin{proof}
We have
\[A(X^m)\equiv \sum_{j\in A}X^{jm}\equiv \sum_{\substack{j\in A\\ jm\equiv g\bmod N}} X^g\bmod(X^N-1),\]
and the last one is precisely the definition of the mask polynomial of the multiset $m\cdot A$. If $m\mid N$, then 
\[\sum_{\substack{j\in A\\ jm\equiv g\bmod N}} X^g\equiv\sum_{\substack{j\in A\\ j\equiv gm^{-1}\bmod N/m}}X^g\equiv 
\sum_{j=0}^{N/m-1}\abs{A_{j\bmod N/m}}X^{jm}\bmod(X^N-1),\]
as desired.
\end{proof}

The Fourier transform of functions on $\ZZ_N$ is defined as
\[\hat{f}(x)=\sum_{d\in\ZZ_N} f(d)\ze_N^{-dx}.\]
The values of mask polynomials $A(X)$ on $N$th roots of unity have the following relation with
those of the Fourier transform of $\bs1_A$:
\[\wh{\bs 1}_A(d)=A(\ze_N^{-d}).\]

The reduced residues $\bmod N$ will be denoted as usual by $\ZZ_N^{\star}$. The group $\ZZ_N$ is partitioned into the subsets 
\begin{equation}\label{divclass}
d\ZZ_N^{\star}=\set{x\in\ZZ_N:\gcd(x,N)=d}=\set{x\in\ZZ_N:\ord(x)=N/d},
\end{equation}
called \emph{divisor classes}, where $d$ runs through the divisors of $N$. Since we will evaluate polynomials on $N$th roots of unity, we unavoidably have to use the cyclotomic polynomials $\Phi_d(X)$ for
$d\mid N$, which is the irreducible polynomial of $\ze_d$ over $\QQ$, and is factored as
\[\Phi_d(X)=\prod_{g\in\ZZ_d^{\star}}(X-\ze_d^g).\]
Therefore, divisibility of $A(X)\bmod(X^N-1)$ by cyclotomic polynomials of degree $d\mid N$, denoted by $\Phi_d(X)$, also makes sense as 
$X^N-1=\prod_{d\mid N}\Phi_d(X)$. $\Phi_d(X)\mid A(X)$ would then hold precisely
when $A(\ze_d)=0$, which would also imply that $A(\ze_d^g)=0$ for every $g\in\ZZ_N$ with $\gcd(d,g)=1$. From this, we readily deduce that the zero set
\[Z(A):=Z(\wh{\bs 1}_A)=\set{x\in\ZZ_N:A(\ze_N^{-x})=0}\]
is a union of divisor classes.

\bigskip\bigskip

\section{Tiles and spectral subsets}
\bigskip\bigskip

These notions are similarly defined when we work on a finite cyclic group.

\begin{defn}
 Let $A\ssq\ZZ_N$. We say that $A$ \emph{tiles} $\ZZ_N$ by translations, if there is some $T\ssq\ZZ_N$ such that any element of $\ZZ_N$ can be written uniquely as a sum of an element of $A$ and an element of $T$.
 $T$ is called the \emph{tiling complement} of $A$ and we write $A\oplus T=\ZZ_N$; we will also call $(A,T)$ a \emph{tiling pair}.
 We call $A$ \emph{spectral}, if there is some $B\ssq\ZZ_N$ with $\abs{A}=\abs{B}$ and the exponential functions 
 \[e_b(x)=\ze_N^{bx}\]
 are pairwise orthogonal on $A$, that is,
 \begin{equation}\label{orth}
\scal{e_b,e_{b'}}_A=\sum_{a\in A}\ze_N^{a(b-b')}=0,
\end{equation}
for every $b,b'\in B$, $b\neq b'$, or equivalently, the matrix
\[\frac{1}{\sqrt{\abs{A}}}\bra{\ze_N^{ab}}_{a\in A, b\in B}\]
is unitary.
 In this case, $B$ is called a \emph{spectrum} of $A$, and $(A,B)$ a \emph{spectral pair} of $\ZZ_N$.
\end{defn}
We note that \eqref{orth} is equivalent to the following matrix being unitary:



Using mask polynomials is very useful when describing tiling properties.
The following was proven in \cite{CM} (Lemma 1.3 and 3.1).

\begin{lemma}\label{tilingpair}
 Let $A,T$ be multisets in $\ZZ_N$. Then, $A$ and $T$ are proper 
 sets and form a tiling pair if and only if
 \[A(X)T(X)\equiv 1+X+\dotsb+X^{N-1}\bmod(X^N-1).\]
 Furthermore, if $m$ is prime to $\abs{A}$, the above equation also implies
 \[A(X^m)T(X)\equiv 1+X+\dotsb+X^{N-1}\bmod(X^N-1),\]
 i.e. $m\cdot A$ is also a tiling complement of $T$.
\end{lemma}

\begin{rem}
Lemma 3.1 actually proves the second part when $m=p$ a prime; however, one could use this repetitively for all primes dividing $m$ and obtain
the above result.
\end{rem}

Spectrality and the tiling property impose some structural properties on the difference set $A-A$.

\begin{thm}\label{mainref}
 Let $A\ssq\ZZ_N$.
 \begin{enumerate}[{\bf(i)}] 
  \item If $A$ has a spectrum $B\ssq\ZZ_N$, if and only if $\abs{A}=\abs{B}$ and $B-B\ssq\set{0}\cup Z(A)$. \label{i}
  \item $A$ tiles $\ZZ_N$ by $T\ssq\ZZ_N$, if and only if $N=\abs{A}\abs{T}$ and $(A-A)\cap(T-T)=\set{0}$. \label{ii}
  \item If $(A,B)$ is a spectral pair of $\ZZ_N$ and $m\in\ZZ_N^{\star}$, then $(A,mB)$ is also a spectral pair.
 \end{enumerate}
\end{thm}

\begin{proof}
\begin{enumerate}[{\bf(i)}]
\item This almost follows by definition; the space of functions $f:A\to \CC$ has dimension $\abs{A}$, therefore a spectrum must
have the same cardinality. Furthermore, by \eqref{orth} we get that any $b,b'\in B$ with $b\neq b'$ must satisfy $b-b'\in Z(A)$.
\item Since every element of $\ZZ_N$ can be expressed uniquely as $a+t$ with $a\in A$, $t\in T$, we get a bijection between $\ZZ_N$ and
$A\times T$, proving $N=\abs{A}\abs{T}$. Furthermore, suppose that $a,a'\in A$ and $t,t'\in T$ satisfy $a-a'=t-t'$. Then, $a+t'=a'+t$,
so by uniqueness of representation of $a+t'$ as an element of $A$ and an element of $T$, we must have $a=a'$ and $t=t'$, proving the desired fact.
\item We will show that $m(b-b')\in Z(A)$ for every $m\in\ZZ_N^{\star}$ and $b,b'\in B$ with $b\neq b'$. Assume that $b-b'\in d\ZZ_N^{\star}$,
hence by \eqref{orth} we get $A(\ze_N^{dg})=0$, where $b-b'\equiv dg\bmod N$, with $g\in\ZZ_N^{\star}$. Then, $\Phi_{N/d}(X)\mid A(X)$, thus
\[0=A(\ze_N^{dgm})=A(\ze_N^{m(b-b')}),\]
completing the proof.\qedhere
\end{enumerate}
\end{proof}

Another way to express Theorem \ref{mainref}\ref{i} is the following.

\begin{cor}\label{specord}
 Let $(A,B)$ be a spectral pair in $\ZZ_N$. Then, $A(\ze_{\ord(b-b')})=0$, for all $b,b'\in B$ with $b\neq b'$.
\end{cor}

\begin{proof}
 This follows from Theorem \ref{mainref}\ref{i}, as $Z(A)$ contains all divisor classes $d\ZZ_N^{\star}$ for which there are $b,b'\in B$ satisfying $b-b'\in d\ZZ_N^{\star}$, so by \eqref{divclass}
 $\ord(b-b')=N/d$, or equivalently, $\ze_N^{b-b'}$ is a primitive $N/d$-th root of unity, yielding
 \[A(\ze_{N/d})=A(\ze_{\ord(b-b')})=0,\]
 as desired.
\end{proof}

Denote by $S^N$ be the set of prime powers dividing $N$, and let $S_A^N=\set{s\in S^N: A(\ze_s)=0}$. 
We define the following properties:

\begin{enumerate}[{\bf(T1)}]
 \item $A(1)=\prod_{s\in S_A^N} \Phi_s(1)$ \label{t1}
 \item Let $s_1,s_2,\dotsc,s_k\in S_A^N$ be powers of different primes. Then $\Phi_{s}(X)\mid A(X)$, where $s=s_1\dotsm s_k$. \label{t2}
\end{enumerate}

These properties implicitly assume that the cyclic group under question has order $N$ (see also \cite{KolMat}). 
Sometimes, a set $A\ssq\ZZ_N$ might be reduced $\bmod M$, for some $M\mid N$, $M<N$, when no two elements of $A$ are
congruent $\bmod M$. In this case, we will explicitly mention the underlying group and say for example, that ``$A$ satisfies \ref{t1} and \ref{t2} in $\ZZ_M$''. For subsets of $\ZZ$, these properties were first defined
by Coven and Meyerowitz \cite{CM}, to study tiling subsets of integers. They proved that \ref{t1} and \ref{t2} (for integers) imply that $A$ tiles $\ZZ$; furthermore, if $\abs{A}$ is divided by at most two distinct
primes, then \ref{t1} and \ref{t2} are equivalent to tiling. On the other hand, \L{}aba \cite{Laba} proved that it also
implies that $A$ is spectral. Here, we will show that the same situation carries to the finite cyclic case\footnote{Although \ref{I} and \ref{II} can easily be proven using the arguments for subsets
of integers, we provide a self-contained proof here. In particular, the spectrum constructed in \eqref{labaspectrum} is very similar to the one constructed in the proof of Theorem 1.5 in \cite{Laba}.}.

\begin{thm}\label{LCM}
 Let $A\ssq\ZZ_N$.
 \begin{enumerate}[{\bf(I)}]
  \item If $A$ satisfies \ref{t1} and \ref{t2}, then it tiles $\ZZ_N$. \label{I}
  \item If $A$ satisfies \ref{t1} and \ref{t2}, then it has a spectrum. \label{II}
  \item If $N$ is divisible by at most two distinct primes, then $A$ tiles $\ZZ_N$ if and only if it satisfies \ref{t1} and \ref{t2}. \label{III}
 \end{enumerate}
\end{thm}

\begin{proof}
 Suppose that $A\ssq\ZZ_N$ satisfies \ref{t1} and \ref{t2}. Denote by $N_s$ the maximal divisor of $N$ that is prime to $s$. For example, if
 \[N=p_1^{\al_1}p_2^{\al_2}\dotsm p_k^{\al_k},\]
 then 
 \[N_{p_1^m}=p_2^{\al_2}\dotsm p_k^{\al_k}, \;\;\;\; \forall m \text{ with } 1\leq m\leq \al_1.\]
 Consider the polynomial
 \[T(X)=\prod_{s\in S^N\sm S_A^N}\Phi_s(X^{N_s}).\]
 By construction, $S_T^N=S^N\sm S_A^N$ and $\abs{A}\abs{T}=N$. Next, we will show that $A(X)T(X)$ vanishes on any $N$th root of unity. Let $X=\ze_d$, $d\mid N$. We write
 \[d=p_1^{\be_1}\dotsm p_k^{\be_k},\]
 with $0\leq \be_j\leq \al_j$, $1\leq j\leq k$. If $s=p_j^{\be_j}\in S_T^N$ for some $j$, so that $d=s\cdot d_s$, then $\Phi_s(X^{N_s})$ divides $T(X)$ and $d_s$ divides $N_s$, therefore
 \[\Phi_s(\ze_d^{N_s})=\Phi_s(\ze_s^g)=0,\]
 for $g=N_s/d_s$, which is prime to $s$, yielding $T(\ze_d)=0$. Otherwise, every $p_j^{\be_j}\in S_A^N$, and since $A$ satisfies \ref{t2}, we get $\Phi_d(X)\mid A(X)$, or equivalently $A(\ze_d)=0$. Combined with 
 $\abs{A}\abs{T}=N$, this implies
 \[A(X)T(X)\equiv1+X+\dotsb+X^{N-1}\bmod(X^N-1),\]
 thus $T(X)$ is the mask polynomial of a set $T\ssq\ZZ_N$, which is a tiling complement of $A$, proving \ref{I}. 
 
 For \ref{II}, we define the set $B\ssq\ZZ_N$ consisting of all elements of the form
 \begin{equation}\label{labaspectrum}
 \sum_{s\in S_A^N}k_s\frac{N}{s},
 \end{equation}
 where for each $s\in S_A^N$, $k_s$ ranges through the set $\set{0,1,\dotsc,p-1}$, where $p=\Phi_s(1)$ is the prime dividing $s$. Different choices of $k_s$, $s\in S_A^N$ give different elements, yielding $\abs{A}=\abs{B}$,
 since $A$ satisfies \ref{t1}. Moreover, an element in $B-B$ has the form \eqref{labaspectrum}, with the difference that $\abs{k_s}<\Phi_s(1)$ holds instead of $0\leq k_s\leq \Phi_s(1)-1$. The order of $k_s\frac{N}{s}$ with
 $\abs{k_s}<\Phi_s(1)$ is exactly $s$. Based on this simple fact, the order of a nonzero element of $B-B$ is $s_1s_2\dotsm s_{\ell}$, where $s_1,\dotsc,s_{\ell}\in S_A^N$ powers of distinct primes. Since $A$ satisfies
 \ref{t2}, we will then have $A(\ze_{\ord(b-b')})=0$ for every $b,b'\in B$, $b\neq b'$, so $B$ is the spectrum of $A$ by Theorem \ref{mainref}\ref{i} and Corollary \ref{specord}.
 
 A proof for part \ref{III} is in \cite{MK}.
\end{proof}

Using \ref{II} and \ref{III} we can easily deduce:

\begin{thm}\label{tspmqn}
 If $A$ tiles $\ZZ_N$, where $N$ is divisible by at most two distinct primes, then $A$ has a spectrum.
\end{thm}

\bigskip\bigskip
\section{Proof of Theorem \ref{mainthm2}}
\bigskip\bigskip

Here, we will test the tools of the previous section, to prove the Tiling$\Rar$Spectral direction in $\ZZ_N$, where $N=p_1^mp_2p_3\dotsm p_n$, where $p_1,\dotsc,p_n$ are distinct primes, and $m,n\geq2$ are integers.
We will do so by proving that \ref{t1} and \ref{t2} always hold for a tiling subset of $\ZZ_N$.

Suppose that $A\ssq\ZZ_N$ tiles $\ZZ_N$ by translations, so that $A\oplus T=\ZZ_N$. If $\gcd(\abs{A},\abs{T})=1$, then Lemma \ref{tilingpair} implies that
\begin{equation}\label{tilingM}
A(X)T(X^M)\equiv 1+X+\dotsb+X^{N-1}\bmod(X^N-1),
\end{equation}
where $M=\abs{A}$. Therefore, $M\cdot T$ is a proper set; moreover, $\abs{M\cdot T}=\abs{T}=N/M$ and $M\cdot T\ssq M\ZZ_N$, which yield $M\cdot T=M\ZZ_N$, i.e. $A$ tiles by the subgroup $M\ZZ_N$, or equivalently,
\[A(X)\equiv 1+X+\dotsb+X^{M-1}=\prod_{1<d\mid M}\Phi_d(X)\bmod(X^M-1).\]
From the last representation of $A(X)$ as a product of cyclotomic polynomials $\bmod(X^M-1)$, it is evident that $A$ satisfies \ref{t1} and \ref{t2} in $\ZZ_N$ and hence is a spectral subset of $\ZZ_N$.
We remark that the direction Tiling$\Rar$Spectral holds for every tiling subset $A\ssq\ZZ_N$ satisfying $\gcd(\abs{A},N/\abs{A})=1$, regardless of the primes dividing $N$, using the same argument.

We proceed to the difficult case, where $\gcd(\abs{A},\abs{T})>1$, and hence $\gcd(\abs{A},\abs{T})$ is a power of $p_1$. Let $M$ be the maximal divisor of $\abs{A}$ that is prime to $p_1$; in particular, 
$\gcd(M,\abs{T})=1$, hence as before, $M\cdot T$ is a proper set and also a tiling complement of $A$ by Lemma \ref{tilingpair}, so that \eqref{tilingM} holds. Next, suppose that
\[S_A^N=\set{p_1^{\ell_1},\dotsc,p_1^{\ell_r},p_2,\dotsc,p_k},\]
so that
\[S_{M\cdot T}^N=S^N\sm S_A^N=\set{p^{m_1},\dotsc,p^{m_s},p_{k+1},\dotsc},\]
where the sets of exponents $\set{\ell_1,\dotsc\ell_r}$ and $\set{m_1,\dotsc,m_s}$ form a partition of $\set{1,2,\dotsc,m}$, by \eqref{tilingM}. It is clear that $A$ satisfies \ref{t1}, as 
\[\abs{A}\cdot\abs{T}=N, \;\;\;\; p_1^rp_2\dotsm p_k\mid\abs{A}, \;\;\;\; p_1^sp_{k+1}\dotsm p_n\mid\abs{T}\]
imply that $\abs{A}=p_1^rp_2\dotsm p_k$
(besides, it is already known that
tiling always implies \ref{t1} in every $\ZZ_N$ or in $\ZZ$ \cite{CM}). For the property \ref{t2}, consider first $d$ to be a composite divisor of $p_2\dotsm p_k=M$. Then,
$T(\ze_d^M)=T(1)\neq0$, hence $A(\ze_d)=0$ by \eqref{tilingM}, verifying \ref{t2} when we consider primes among $p_2,\dotsc,p_k$. Next, consider a root of unity of order $p_1^{\ell_j}d$ ($d\mid M$ as before); we get
\[T(\ze_{p_1^{\ell_j}d}^M)=T(\ze_{p_1^{\ell_j}}^{M/d})=\sig(T(\ze_{p^{\ell_j}_1}))\neq0,\]
for some $\sig\in\Gal(\QQ(\ze_{p_1^{\ell_j}})/\QQ)$, since $p_1\nmid\frac{M}{d}$, whence
\[\Phi_{p^{\ell_j}_1d}(X)\mid A(X),\]
confirming \ref{t2} completely for $A$. By Theorem \ref{LCM}\ref{II}, $A$ is spectral, completing the proof of Theorem \ref{mainthm2}.

\bigskip\bigskip

\section{Vanishing sums of roots of unity}
\bigskip\bigskip

Let $G=\scal{\ze_N}$, the cyclic group generated by the standard $N$th root of unity. Consider the group ring $\ZZ[G]$, the ring of all formal integer linear combinations of $N$th roots of unity. It is known
that $\ZZ[G]\cong\ZZ[X]/(X^N-1)$, so we may represent an element of $\ZZ[G]$ by a polynomial with integer coefficients modulo $X^N-1$. $\ZZ[G]$ is equipped with a natural evaluation map, say $v$, which simply
evaluates the given sum. Considered as an element of $\ZZ[X]/(X^N-1)$, say $F(X)$, the effect of the evaluation map is $v(F)=F(\ze_N)$, as expected.

When this sum vanishes, we obtain some information on the structure of $F(X)$. We note that $\Phi_p(\ze_p)=0$ for $p$ prime, implies that
\[1+\ze_p+\dotsb+\ze_p^{p-1}\]
is a vanishing sum of roots of unity; the same holds true for any integer linear combination of sums of this form, multiplied by roots of unity. The following theorem, attributed to
R\'edei \cite{Redei50,Redei54}, de Bruijn \cite{deB53}, and Schoenberg \cite{Sch64}, shows that the converse is true as well. We will express it in polynomial notation.

\begin{thm}\label{vansums}
 Let $F(X)\in\ZZ[X]$ with $F(\ze_N)=0$. Then,
 \[F(X)\equiv \sum_{p\mid N, \; p \text{ prime}}F_p(X)\Php\bmod(X^N-1),\]
 where $F_p(X)\in\ZZ[X]$.
\end{thm}

When $N$ is divisible by at most two primes and $F(X)$ has nonnegative coefficients, we obtain something stronger. This is due to Lam and Leung \cite{LL}.

\begin{thm}\label{vansumspos}
 Let $F(X)\in\ZZ_{\geq0}[X]$. Then, $F(\ze_N)=0$, where $N=p^mq^n$ and $p, q$ primes, if and only if 
 \[F(X)\equiv P(X)\Php+Q(X)\Phq\bmod(X^N-1),\]
 for some $P(X),Q(X)\in\ZZ_{\geq0}[X]$. If $F(\ze_N^{p^k})\neq0$ (respectively, $F(\ze_N^{q^{\ell}})\neq0$) for some $1\leq k\leq m$ (resp. $1\leq \ell\leq n$), then we cannot have $P(X)\equiv0\bmod(X^N-1)$
 (resp. $Q(X)\equiv0\bmod(X^N-1)$).
\end{thm}

This is no longer true if $N$ has at least three distinct prime factors, say $p,q,r$, as the polynomial
\[F(X)=(\Php-1)(\Phq-1)+(\Phi_r(X^{N/r})-1)\]
satisfies $F(\ze_N)=0$, but has no such representation.

We observe that 
\[\Php=1+X^{N/p}+X^{2N/p}+\dotsb+X^{(p-1)N/p}\] 
is the mask polynomial of the subgroup of $\ZZ_N$ of order $p$; in general, for $d\mid N$, the cosets of the subgroup of order $d$, $\frac{N}{d}\ZZ_N$, will be called
$d$-cycles. So, Theorem \ref{vansumspos} asserts that a vanishing sum of roots of unity of order $N=p^mq^n$ can be expressed as a \emph{disjoint} union (i.e. as multisets, counting
multiplicities) of $p$- and $q$-cycles. This is very important when we consider the difference set $A-A$ of a spectral set $A\ssq\ZZ_N$ with $A(\ze_N)=0$, as it yields both $\frac{N}{p}\ZZ_N$, $\frac{N}{q}\ZZ_N\ssq A-A$,
under certain mild conditions, which in turn
gives roots for the mask polynomial of the spectrum $B$, due to Theorem \ref{mainref}\ref{i}.

The condition $A(\ze_N)=0$ is satisfied when a spectrum $B$ is \emph{primitive}; we will subsequently focus on spectral pairs $(A,B)$ of primitive sets, otherwise, we may consider $A$ and $B$ (after an appropriate
translation) as subsets of a proper subgroup.

\begin{defn}
  Let $G$ be a finite Abelian group and $T$ a subset. We call $T$ primitive, if it is not contained in a coset of a proper subgroup of $G$.
 \end{defn}
 
 \begin{lemma}\label{primitive}
  Let $G=\ZZ_N$ with $N=p^m q^n$. A subset $T\ssq G$ is primitive if and only if $(T-T)\cap\ZZ_N^{\star}\neq\vn$.
 \end{lemma}
 \begin{proof}
  Since $T$ is primitive, for a given $t\in T$ it would hold $t-T\nssq p\ZZ_N$ and $t-T\nssq q\ZZ_N$, yielding $t',t''\in T$ such that $t-t'\notin p\ZZ_N$, $t-t''\notin q\ZZ_N$. If either $t-t'\notin q\ZZ_N$
  or $t-t''\notin p\ZZ_N$, then we get $t-t'\in\ZZ_N^{\star}$ or $t-t''\in\ZZ_N^{\star}$, respectively. Otherwise, $t-t'\in q\ZZ_N$ and $t-t''\in p\ZZ_N$, therefore $t'-t''\notin p\ZZ_N\cup q\ZZ_N$, giving
  $t'-t''\in\ZZ_N^{\star}$, thus $(T-T)\cap\ZZ_N^{\star}\neq \vn$. The converse is trivial.
 \end{proof}

\begin{prop}\label{preroothunt}
 Let $A\ssq\ZZ_N$. Suppose that there is $j$ such that $A\cap(j+\frac{N}{pq}\ZZ_N)=A_{j\bmod\frac{N}{pq}}$ is not supported on a $p$- or $q$-cycle. Then $(A-A)\cap\frac{N}{pq}\ZZ_N^{\star}\neq\vn$.
\end{prop}

\begin{proof}
 Let $j$ be as in the hypothesis, and without loss of generality assume that $A\ssq j+\frac{N}{pq}\ZZ_N$, or equivalently, $A=A_{j\bmod\frac{N}{pq}}$. 
 Translating $A$ does not affect the conclusion, so we may further assume that $j=0$. This confines
 $A$ in the subgroup of order $pq$; dividing every element by $\frac{N}{pq}$ we can assume that $N=pq$. The problem has thus been reduced to the following: assume that $A\ssq\ZZ_{pq}$ is not a subset of a $p$- or a
 $q$-cycle; then $(A-A)\cap\ZZ_{pq}^{\star}\neq\vn$.  Indeed, this is the case, since $A$ is primitive, so the conlcusion of the reduced problem follows from Lemma \ref{primitive} with $m=n=1$.
\end{proof}

\begin{lemma}\label{roothunt}
 Let $A\ssq\ZZ_N$ with $N=p^mq^n$. Suppose that there is $d\mid \frac{N}{q}$ such that $(qd)\cdot A$ is a union of $p$-cycles, but $A(\ze_N^d)\neq0$. Then,
 \[(A-A)\cap\frac{N}{pqd}\ZZ_N^{\star}\neq\vn.\]
 If in addition, $A$ is spectral, then $B(\ze_{pqd})=0$, for any spectrum $B$ of $A$.
\end{lemma}

\begin{proof}
 By hypothesis and Proposition \ref{maskpoly}, we have
 \[A(X^{qd})\equiv P(X)\Php\bmod(X^N-1)\]
 Consider the intersections of $dA$ with the cosets of the subgroup of index $pq$, namely
 \[dA\cap(j+\frac{N}{pq}\ZZ_N)=\bra{dA}_{j\bmod\frac{N}{pq}},\]
 where $0\leq j\leq\frac{N}{pq}-1$. If $\bra{dA}_{j\bmod\frac{N}{pq}}\neq\vn$, then it intersects every coset of the subgroup $\frac{N}{q}\ZZ_N$ contained in $j+\frac{N}{pq}\ZZ_N$. 
 Indeed, if $da\in \bra{dA}_{j\bmod\frac{N}{pq}}$, then
 \[qda+i\frac{N}{p}\in qdA,\]
 for $0\leq i\leq p-1$ by hypothesis, so, for each $i$ there is $a'\in A$ such that $qda'\equiv qda+i\frac{N}{p}\bmod N$, therefore,
 \[da'\equiv da+i\frac{N}{pq}+k\frac{N}{q}\bmod N,\]
 for some $k$. Each element of the form $da+i\frac{N}{pq}+k\frac{N}{q}$ belongs to a different $q$-cycle of $j+\frac{N}{pq}\ZZ_N$, and since there are exactly $p$ of them, $dA$ intersects them all. If every pair
 of elements $da,da'\in \bra{dA}_{j\bmod\frac{N}{pq}}$ lying on two different $q$-cycles, belongs to the same $p$-cycle, then it is evident that $\bra{dA}_{j\bmod\frac{N}{pq}}$ is a single $p$-cycle. By hypothesis,
 the multiset $(qd)\cdot A$ contains the entire $p$-cycle $qj+\frac{N}{p}\ZZ_N$ with some multiplicity (i.e. every element in this $p$-cycle appears with the same positive multiplicity in $(qd)\cdot A$),
 since $\bra{dA}_{j\bmod\frac{N}{pq}}\neq\vn$. We have proven that the elements $d\cdot A$ restricted to $j+\frac{N}{pq}\ZZ_N$ (i.e. the inverse image of $qj+\frac{N}{p}\ZZ_N$ under the multiplication by $q$ map) 
 are supported on
 a single $p$-cycle; therefore, this $p$-cycle must appear with the same multiplicity in $d\cdot A$ as $qj+\frac{N}{p}\ZZ_N$ appears in $(qd)\cdot A$. Thus, if every $\bra{dA}_{j\bmod\frac{N}{pq}}$
 is supported on a $p$-cycle, we
 conlcude that $d\cdot A$ is also a union of $p$-cycles as a multiset, giving $A(\ze_N^d)=0$, a contradiction.
 
 This proves the existence of
 $da,da'\in \bra{dA}_{j\bmod\frac{N}{pq}}$, lying on two different $q$-cycles, and two different $p$-cycles as well, for some $j$. Then $\bra{dA}_{j\bmod\frac{N}{pq}}$
 is not supported on a $p$- or $q$-cycle, so by Proposition \ref{preroothunt}
 \[(dA-dA)\cap\frac{N}{pq}\ZZ_N^{\star}\neq\vn,\]
 or equivalently, there are $a,a'\in A$ such that $d(a-a')\in\frac{N}{pq}\ZZ_N^{\star}$, which shows that $\ord(a-a')=pqd$; indeed, as $pqd(a-a')\equiv 0\bmod N$, while $pd(a-a')\in\frac{N}{q}\ZZ_N^{\star}$
 and $qd(a-a')\in\frac{N}{p}\ZZ_N^{\star}$. This implies that $(A-A)\cap\frac{N}{pqd}\ZZ_N^{\star}\neq\vn$, hence $B(\ze_{pqd})=0$ for any spectrum $B$ of $A$ by Corollary \ref{specord}, completing the proof.
\end{proof}

 \bigskip\bigskip
 
 \section{Some special cases}
 \bigskip\bigskip
 
 For the rest of this article, we restrict to $N=p^mq^n$, where $p,q$ are distinct primes. There are some special cases that yield somewhat easily the Spectral$\Rar$Tile direction, and serve as building blocks
 in more general cases.

Consider the natural map $\ZZ[X]\twoheadrightarrow\FF_q[X]$, which is just reduction of the coefficients $\bmod q$. The image of a polynomial $P(X)\in\ZZ[X]$ will be denoted simply by $\ol{P}(X)$.
Reducing the coefficients modulo a prime dividing $N$ seems very advantageous when a small power of said prime appears in the prime decomposition of $N$.
 
 \begin{thm}\label{nopq}
  Let $p$, $q$ be distinct primes and $N=p^mq^n$, where $m$, $n$ are positive integers. Suppose that $A\ssq\ZZ_N$ is spectral, such that $pq\nmid\abs{A}$. Then $A$ tiles $\ZZ_N$ by translations.
 \end{thm}
 
 \begin{proof}
  Without loss of generality, we may assume that $q\nmid\abs{A}$. 
  It holds $A(\ze)\neq0$ when $\ze^{q^n}=1$;
  otherwise, $A(X)$ would be divided by some cyclotomic polynomial $\Phi_{q^k}(X)$ for some $1\leq k\leq n$ and as a consequence, $q=\Phi_{q^k}(1)$ would divide $A(1)=\abs{A}$, a contradiction. Next,
  consider the set
  \[R=\set{r\in[1,m]:A(\ze_{p^r}\ze_{q^k})=0,\text{ for some }k, 0\leq k\leq n}.\]
  Define the following polynomial
  \begin{equation}\label{polyF}
  F(X)=A(X)\prod_{\substack{1\leq r\leq m\\ r\notin R}}\Phi_{p^r}(X),
  \end{equation}
  which has nonnegative integer coefficients. We will show that $F(1)\leq p^m$, or equivalently, $A(1)\leq p^{\abs{R}}$. To this end, consider the spectrum $B$ of $A$, which satisfies $\abs{A}=\abs{B}$
  and $B-B\ssq\set{0}\cup Z(A)$ by Theorem \ref{mainref}\ref{i}.
  Since $A(\ze)\neq0$ when $\ze^{q^n}=1$, we obtain $(B-B)\cap p^m\ZZ_N=\set{0}$, therefore, any two distinct elements of $B$ are not congruent modulo $p^m$ (giving a first estimate of $\abs{A}\leq p^m$).
  So, we may write for every $b\in B$
  \[b\equiv b_0+b_1p+\dotsb+b_{m-1}p^{m-1}\bmod p^m,\	0\leq b_j\leq p-1,\	0\leq j\leq m-1,\]
  and observe that no two elements of $B$ correspond to the same $m$-tuple $(b_0,\dotsc,b_{m-1})$. Moreover, no two elements of $B$ can agree on all $p$-adic digits $b_r$, for $m-r\in R$; if such elements existed,
  say $b\neq b'$, then $p^r\pdiv b-b'$ for some $m-r\notin R$ and since $(A,B)$ is a spectral pair, we should have $A(\ze_{p^{m-r}}\ze_{q^k})=0$ for some $k$
	by Corollary \ref{specord}, contradicting the fact $m-r\notin R$. This
  clearly shows that $\abs{A}\leq p^{\abs{R}}$, and thus $F(1)\leq p^m$.
  
  Next, we will show that $\ol{\Phi}_{p^r}(X)\mid \ol{F}(X)$ in $\FF_q[X]$ for $1\leq r\leq m$. If $r\notin R$ this follows immediately from \eqref{polyF}. If $r\in R$, then by definition $\Phi_{p^rq^k}(X)\mid A(X)$
  for some $k$,
  hence $\Phi_{p^rq^k}(X)\mid F(X)$, where $p^rq^k\mid N$. Using standard properties of cyclotomic polynomials we get
  \[\Phi_{p^rq^k}(X)=\Phi_{pq}(X^{p^{r-1}q^{k-1}})=\frac{\Phi_p(X^{p^{r-1}q^k})}{\Phi_p(X^{p^{r-1}q^{k-1}})}.\]
  Reducing coefficients $\bmod q$, we obtain
  \[\frac{\ol{\Phi}_p(X^{p^{r-1}})^{q^k}}{\ol{\Phi}_p(X^{p^{r-1}})^{q^{k-1}}}=\ol{\Phi}_p(X^{p^{r-1}})^{q^{k-1}(q-1)}=\ol{\Phi}_{p^r}(X)^{q^{k-1}(q-1)}.\]
  Therefore,
  \[\prod_{r=1}^m\ol{\Phi}_{p^r}(X)=X^{p^m-1}+\dotsb+1\mid \ol{F}(X)\]
  in $\FF_q[X]$, say
  \[\ol{F}(X)=(X^{p^m-1}+\dotsb+1)G(X).\]
  Now, we reduce $\bmod(X^{p^m}-1)$. Since
  \[X^{\ell}(X^{p^m-1}+\dotsb+1)\equiv X^{p^m-1}+\dotsb+1\bmod(X^{p^m}-1)\]
  for every $\ell$, we get
  \begin{equation}\label{Freduced}
  \ol{F}(X)\equiv G(1)(X^{p^m-1}+\dotsb+1)\bmod(X^{p^m}-1),
  \end{equation}
  where $G(1)\neq0$ (in $\FF_q$) by assumption. Consider the remainder of the Euclidean division $F(X):X^{p^m}-1$ in $\ZZ[X]$, say $R(X)$. It is not hard to show that $R$ has also nonnegative coefficients; indeed,
  if $F(X)=\sum_{k=0}^{\deg F}f_kX^k\in\ZZ_{\geq0}[X]$, then
  \[R(X)=\sum_{k=0}^{p^m-1}\bra{\sum_{\ell\equiv k\bmod(p^m-1)}f_\ell}X^k\in\ZZ_{\geq0}[X].\]
  All coefficients of $R(X)$
  must be congruent to $c\equiv G(1)\bmod q$ due to \eqref{Freduced}. Without loss of generality, we may assume $0<c<q$, hence positivity of the coefficients of $R$ implies
  \[R(X)=c(X^{p^m-1}+\dotsb+1)+qQ(X),\]
  where $Q(X)\in\ZZ_{\geq0}[X]$, therefore,
  \[F(1)=R(1)=cp^m+qQ(1)\geq p^m,\]
  implying $F(1)=p^m$ and $Q(X)\equiv 0$. From this we obtain 
  \[F(X)\equiv X^{p^m-1}+\dotsb+1\bmod(X^{p^m}-1)\]
  in $\ZZ[X]$. Clearly,
  \[F(X)\sum_{s=0}^{q^n-1}X^{p^ms}\equiv X^{N-1}+\dotsb+1\bmod(X^N-1),\]
  so, for the polynomial $C(X)\in\ZZ[X]$ with nonnegative coefficients defined by
  \[C(X)=\bra{\sum_{s=0}^{q^n-1}X^{p^ms}}\prod_{\substack{1\leq r\leq m\\ r\notin R}}\Phi_{p^r}(X)\]
  we have
  \[A(X)C(X)\equiv X^{N-1}+\dotsb+1\bmod(X^N-1),\]
  implying that $C(X)$ is the mask polynomial of a set $C$ which is the tiling complement of $A$. This completes the proof.
 \end{proof}
 
 Theorem \ref{nopq} proves the Spectral$\Rar$Tiling direction for sets $A$ whose cardinality is prime to either $p$ or $q$. On the other extreme, the same holds when the cardinality is divisible by the maximal
 power of either $p$ or $q$, dividing $N$.

\begin{prop}\label{maxpower}
Let $N=p^mq^n$ and $A\ssq\ZZ_N$ spectral such that $p^m\mid \abs{A}$ or $q^n\mid \abs{A}$. Then $A$ tiles $\ZZ_N$.
\end{prop}

\begin{proof}
Without loss of generality, $p^m\mid \abs{A}$, and let $B$ be a spectrum of $A$. Suppose that
 \[S_B^N=\set{p^{m_1},\dotsc,p^{m_{\ell}},q^{n_1},\dotsc,q^{n_k}},\]
 where $1\leq m_1<\dotsb<m_{\ell}\leq m$ and $1\leq n_1<\dotsb<n_k\leq n$. Since
 \[\Phi_{q^{n_1}}(X)\dotsm\Phi_{q^{n_k}}(X)\mid B(X),\]
 we obtain
 \[p^m q^k\mid \abs{B}=\abs{A}\]
 by putting $X=1$, and using the hypothesis.
 For every $j$ consider $A_{j\bmod p^m}$. Since $A_{j\bmod p^m}-A_{j\bmod p^m}\ssq p^m\ZZ_N$, we must have
 \begin{equation}\label{ajdiff}
 A_{j\bmod p^m}-A_{j\bmod p^m}\ssq\set{0}\cup p^mq^{n-n_1}\ZZ_N^{\star}\cup\dotsb\cup p^mq^{n-n_k}\ZZ_N^{\star}
 \end{equation}
 by Theorem \ref{mainref}\ref{i}.
 As usual, we consider the (finite) $q$-adic expansion of each $a\in\ZZ_N$ as
\[a\equiv a_0+a_1q+a_2q^2+\dotsb+a_{n-1}q^{n-1}\bmod q^n, \;\;\;\; 0\leq a_i\leq q-1, 0\leq i\leq n-1.\]
 Suppose that $\abs{A_{j\bmod p^m}}>q^k$ for some $j$. 
This would imply the existence of $a,a'\in A_{j\bmod p^m}$ that would have the same $q$-adic digits at positions $n-n_1,\dotsc,n-n_k$, hence $a-a'$ would not
 belong to the above union, contradiction. Thus, $\abs{A_{j\bmod p^m}}\leq q^k$ for all $j$, and in light of $p^mq^k\mid\abs{A}$ we must have
 \[\abs{A_{j\bmod p^m}}=q^k,\	\forall j.\]
 This further implies $A(\ze_{p^{\ell}})=0$, $1\leq\ell\leq m$. Using the same argument as above, where we reverse
the roles of $A$ and $B$, we also obtain $B(\ze_{p^{\ell}})=0$, for $1\leq \ell\leq m$, whence
\[\abs{B_{j\bmod p^m}}=q^k,\	\forall j.\]
Next, consider
\[S_A^N=\set{p,p^2,\dotsc,p^m,q^{s_1},\dotsc,q^{s_t}},\]
where $1\leq s_1<\dotsb<s_t\leq n$, so that
\[\Phi_{q^{s_1}}(X)\dotsm\Phi_{q^{s_t}}(X)\mid A(X),\]
and putting $X=1$, we get $q^t\mid\abs{A}$, which yields $t\leq k$. We actually have $t=k$, since
 \begin{equation}\label{bjdiff}
B_{j\bmod p^m}-B_{j\bmod p^m}=\set{0}\cup p^mq^{n-s_1}\ZZ_N^{\star}\cup\dotsb\cup p^mq^{n-s_t}\ZZ_N^{\star},
\end{equation}
for every $j$, and $\abs{B_{j\bmod p^m}}=q^k$; if $t<k$, then there would exist $b,b'\in B_{j\bmod p^m}$, $b\neq b'$ whose $q$-adic expansions
$\bmod q^n$ agree on all $q$-adic digits corresponding to $q^{n-s_i}$, $1\leq i\leq t$. But then
\[b-b'\notin \set{0}\cup p^mq^{n-s_1}\ZZ_N^{\star}\cup\dotsb\cup p^mq^{n-s_t}\ZZ_N^{\star},\]
contradiction. Thus, $t=k$ and in particular, $A$ satisfies \ref{t1}. 
 
Next, consider the polynomial
\[\La(X)=\prod_{i=1}^k\Phi_{q^{n_i}}(X^{p^m}).\]
As the coefficients of $\La(X)$ are $0$ or $1$, it is the mask polynomial of a set $\La\ssq\ZZ_N$ with cardinality $\abs{\La}=\La(1)=q^k$.
By \eqref{ajdiff}, $\La$ is then the common spectrum of the sets $A_{j\bmod p^m}$. It holds
\[\bigcup_{i=1}^k p^mq^{n_i-1}\ZZ_N^{\star}\ssq \La-\La \ssq\set{0}\cup Z(A_{j\bmod p^m}),\]
for every $j$ by Theorem \ref{mainref}\ref{i}, therefore,
\[A_{j\bmod p^m}(\ze_{q^{n-n_i+1}})=0,\	1\leq i\leq k,\	 \forall j,\]
thus $A(\ze_{q^{n-n_i+1}})=0$, $1\leq i\leq k$, proving eventually that $n_i+s_i=n+1$, and
\[S_A^N=\set{p,p^2,\dotsc,p^m,q^{n-n_1+1},\dotsc,q^{n-n_k+1}}.\]

Next, we fix $j$ and consider the polynomial $F(X)$ satisfying
 \[A_{j\bmod p^m}(X)\equiv X^j F(X^{p^m})\bmod (X^N-1).\]
 Since $\Phi_{q^{n-n_i+1}}(X)\mid F(X^{p^m})$ for $1\leq i\leq k$, and $p^m$ is prime to $q^{n-n_i+1}$, we also have $\Phi_{q^{n-n_i+1}}(X)\mid F(X)$. Therefore,
 for $1\leq \ell\leq m$ we have
 \[A_{j\bmod p^m}(\ze_{p^{\ell}q^{n-n_i+1}})=\ze_{p^{\ell}q^{n-n_i+1}}^jF(\ze_{p^{\ell}q^{n-n_i+1}}^{p^m})=\ze_{p^{\ell}q^{n-n_i+1}}^jF(\ze_{q^{n-n_i+1}}^{p^{m-\ell}})=0,\]
 proving
 \[A(\ze_{p^{\ell}q^{n-n_i+1}})=0,\	1\leq i\leq k,\	1\leq \ell\leq m,\]
 thus $A$ satisfies \ref{t2} as well, hence tiles $\ZZ_N$ by virtue of Theorem \ref{LCM}\ref{I}, as desired.
\end{proof}
 
 The techniques developed so far cover the result proven in \cite{MK}.
 
 \begin{cor}\label{mkshort}
  Let $N=p^nq$. Then $A\ssq\ZZ_N$ is spectral if and only if it is a tile.
 \end{cor}
 
 \begin{proof}
 If $A$ tiles, then it is spectral by Theorem \ref{tspmqn}. Suppose that $A$ is spectral. If $q\nmid \abs{A}$, then $A$ tiles by Theorem \ref{nopq}; if $q\mid\abs{A}$,
  then $A$ tiles by Proposition \ref{maxpower}.
 \end{proof}

\bigskip\bigskip
 \section{Sketch of the proof of Theorem \ref{mainthm}}
 \bigskip\bigskip

Before proceeding to the most technical part of the proof of Theorem \ref{mainthm}, we will describe the main ideas. First of all, the proof is inductive in nature; so far, there are
some positive results involving products of two prime powers (for example, Corollary \ref{mkshort}), so we may work in a group $\ZZ_N$ with $N=p^mq^n$, such that Fuglede's conjecture
holds in all proper subgroups of $\ZZ_N$, or in other words, it holds in all groups of the form $\ZZ_M$ with $M\mid N$, $M<N$.

If $(A,B)$ is a spectral pair in such a group, we may easily make some reductions for both $A$ and $B$. If either $A$ or $B$ is a subset of a coset of a proper subgroup, we reduce to
a spectral pair in a subgroup of $\ZZ_N$, hence we obtain the tiling property in the smaller group, which in turn implies that $A$ tiles $\ZZ_N$ (Proposition 
\ref{primitivereduction}). Therefore, we reduce to the case where $A$ and every spectrum $B$ are primitive. Primitivity implies that both difference sets $A-A$ and $B-B$
intersect $\znp$ by Lemma \ref{primitive}. Then, Theorem \ref{mainref}\ref{i} yields that $\ze_N$ is a root of both $A(X)$ and $B(X)$ (Corollary \ref{Nroot}), 
hence each of them is a disjoint union
of $p$- and $q$-cycles by Theorem \ref{vansumspos}. We obtain thus a first strong information about the structure of both $A$ and $B$, which in turn shows that $A-A$ and $B-B$
intersect other divisor classes, which provide more roots for $A(X)$ and $B(X)$, hence more information about $A$ and $B$, and so on.

Further reductions are possible; if $A$ is a union only of $p$-cycles (or only of $q$-cycles), then we can again reduce the problem to a smaller subgroup where Fuglede's conjecture
holds, which yields that $A$ tiles $\ZZ_N$. We can thus reduce to the case where $A$ is not a union only of $p$-cycles, or only of $q$-cycles (Proposition \ref{Apq}). The same
fact applies for every spectrum $B$ of $A$, however, it is more difficult to prove (Corollary \ref{Bpq}).

Next, we consider subsets of a spectral set $A\ssq\ZZ_N$, of the form $A_{j\bmod d}=A\cap(j+d\ZZ_N)$, where $d\mid N$. If, say $pd\mid N$, we have the obvious partition of 
$A_{j\bmod d}$
\[A_{j\bmod d}=\bigsqcup_{k=0}^{p-1}A_{j+kd\bmod pd},\]
and a same partition holds when $qd\mid N$. If all sets in the above disjoint union have equal cardinalities, then we say that $A_{j\bmod d}$ is equidistributed $\bmod pd$, and
on the other extreme, if only one of the above sets is nonempty, we say that $A_{j\bmod d}$ is absorbed $\bmod pd$. One or the other property holds when certain combination of roots
of $A(X)$ and $B(X)$ appears (Lemma \ref{abseqd}). Taking into account how $A_{j\bmod d}$ distributes $\bmod pd$ or $\bmod qd$ and then taking the difference between sets in the
union above, we obtain information not only about the roots of $B(X)$ (using Theorem \ref{mainref}\ref{i}), but also of $A(X)$. This is the content of Section \ref{absorptionequid}.

If there is some $d=p^k$ (respectively, $d=q^\ell$), such that either $A_{j\bmod d}$ or $B_{j\bmod d}$ is absorbed $\bmod pd$ (respectively, $\bmod qd$) for every $j$, then we can extend both
sets so that we obtain a new spectral pair $(A',B')$ (Lemma \ref{ascend}, Corollary \ref{multascend}, Lemma \ref{absfreeredux}). Both $A'$ and $B'$ are absorption-free, and $A'=A\oplus S$, for some $S\ssq\ZZ_N$ for which $0\in S$. In particular, if $A'$ tiles, then so does $A$. With this argument, we may further reduce the problem to those spectral sets $A$ that are absorption-free as well. This certainly holds true when $A\in\F(N)$ has maximal size; we denote the set of all such spectral sets by $\Fmax(N)$.

The failure of Fuglede's conjecture in $\ZZ_N$, especially when it holds for all proper groups $\ZZ_N$, imposes some symmetry on the roots of the mask polynomials of the spectral pair $(A,B)$, when $A\in\F(N)$ (Lemma \ref{rootpatterns} and Proposition \ref{maxpqroots}). This result, along with subsequent ones, is obtained by checking how sets of the form $B_{j\bmod d}$ or $A_{j\bmod d}$ distribute $\bmod pd$ and $\bmod qd$ (i.e. especially when we have absorption or equidistribution), when $d$ ranges through the powers $p, p^2,\dots$ or $q, q^2\dotsc$. Arguments of this type are used in the proofs of Lemma \ref{t1fail}, and the more technical Lemmata \ref{Ulb} and \ref{deficitestimate} (in Lemma \ref{Ulb} we also examine the distribution of sets of the form $A_{j\bmod p^{m-r-1}q^{n-1}}$).

Lemma \ref{t1fail} is the first indication that if $A\in\Fmax(N)$, then \ref{t1} fails for every spectrum $B$, so that $B\in\Fmax(N)$ (Corollary \ref{wt1tile}). We recall that every root of the form $\ze_{p^x}$ contributes a factor of $p$ in the cardinality of $B$; Lemma \ref{t1fail} states also that we get a contribution of a factor of $p$ even when $B(\ze_{p^x})\neq0=B(\ze_{p^xq})$. Lemma \ref{Ulb} estimates the number of such $x$, and then Lemma \ref{deficitestimate} estimates the number of $x$ for which the symmetry of root patterns fails, called herein the root deficit. These estimates eventually lead to the desired result, when one of the exponents $m$ or $n$ is small enough, or when $p^{m-2}<q^4$.

 \bigskip\bigskip
 \section{Primitive subsets and an inductive approach}
 \bigskip\bigskip
 
 The existence of a spectral subset $A\ssq\ZZ_N$ that does not tile, imposes a very rigid structure on $A$ that is impossible to appear in a subset of $\ZZ_N$, under certain conditions. This is the main argument
 for proving Fuglede's conjecture in certain cyclic groups of order $N=p^mq^n$. We will usually consider groups such that Fuglede's conjecture holds in all proper subgroups.
 
 \begin{defn}
  For two distinct primes $p$, $q$, define the following set:
  \[\Spq=\set{N=p^mq^n:\ZZ_N\text{ has spectral subsets that do not tile}}.\]
  We distinguish also the minimal elements in terms of division by defining the following subset of $\Spq$:
  \[\Spqmin=\set{N\in\Spq:\forall M\mid N\text{ satisfying }M\neq N, \text{ it holds } M\notin \Spq},\]
  or equivalently, $N\in\Spqmin$ if Fuglede's conjecture is not true in $\ZZ_N$ but holds in every subgroup of $\ZZ_N$. Finally, for each $N\in\Spq$ we define
  \[\F(N)=\set{A\ssq\ZZ_N:A\text{ is spectral but does not tile}}.\]
 \end{defn}

 \begin{rem}
 Fuglede's conjecture holds in $\ZZ_{p^mq^n}$ for every $m,n\geq1$, if and only if $\Spq=\vn$. Throughout this paper, we will usually show that Fuglede's conjecture holds for a specific $\ZZ_N$ (or a family
 of cyclic groups), by showing that $N$ cannot belong to $\Spq$, or equivalently, that $\F(N)=\vn$. 
 We will do so by proving some properties of the sets belonging to $\F(N)$, 
 which will be impossible to hold under some conditions, for example if one of the exponents $m,n$ is relatively small.
 \end{rem}
 

 
 Now, we start proving properties for the sets defined above. We recall that a subset of $\ZZ_N$ is called primitive, if it is not contained in a coset of a proper 
subgroup of $\ZZ_N$.
 
 \begin{prop}\label{primitivereduction}
  Let $A\ssq\ZZ_N$ be a spectral set, $N\in\Spqmin$, such that $A$ is not primitive, or possesses a non-primitive spectrum, $B$. Then $A$ satisfies \ref{t1} and \ref{t2}, or equivalently, $A\notin\F(N)$.
 \end{prop}
 
 \begin{proof}
  Suppose that $A$ has a spectrum $B$ that is not primitive. Then, a translate of $B$ must be contained in a maximal subgroup of $\ZZ_N$, i.e. either $p\ZZ_N$ or
	$q\ZZ_N$. This translate is also a spectrum of $A$, so without loss of generality we may assume $B\ssq p\ZZ_N$. 
	This shows the existence of $B'\ssq\ZZ_N$ with $\abs{B}=\abs{B'}$
  and $B=pB'$. The matrix 
  \[\tfrac{1}{\sqrt{\abs{A}}}\bra{\ze_N^{ab}}_{a\in A, b\in B}=\tfrac{1}{\sqrt{\abs{A}}}\bra{\ze_{N/p}^{ab'}}_{a\in A, b'\in B'}\]
  is unitary by definition, hence no two elements of $A$ leave the same residue $\bmod N/p$ (as well as of $B'$). Hence, $(A,B')$ can be considered as a spectral in $\ZZ_{N/p}$, so $A(X)\bmod(X^{N/p}-1)$ satisfies
  \ref{t1} and \ref{t2} by definition of $\Spqmin$. This in particular shows that $A(\ze_{p^m})\neq0$, otherwise \ref{t1} would fail $\bmod N/p$. Therefore, \ref{t1} and \ref{t2} are satisfied by 
  $A(X)\bmod(X^N-1)$ as well.
  
  Next, suppose that $A$ is not primitive, so that $A\ssq p\ZZ_N$ without loss of generality. As before, $(A',B)$ would be a spectral pair in $\ZZ_{N/p}$, where $A=pA'$, so $A'(X)$ satisfies \ref{t1} and
  \ref{t2} by definition of $\Spqmin$. We have $A(X)\equiv A'(X^p)\bmod(X^N-1)$, hence $A(\ze_p)\neq0$. The following show that $A$ also satisfies \ref{t1} and \ref{t2}:
  \[A'(\ze_{p^kq^{\ell}})=0\Longrightarrow A(\ze_{p^{k+1}q^{\ell}})=0,\]
  and
  \[A'(\ze_{q^k})=0\Longrightarrow A(\ze_{pq^k})=A(\ze_{q^k})=0.\]
  For the latter, consider $t$ such that $pt\equiv1\bmod q^k$. Then,
  \[A(\ze_{q^k}^t)=A'(\ze_{q^k}^{pt})=A'(\ze_{q^k})=0,\]
  and by a Galois action we get $A(\ze_{q^k})=0$ as well. This completes the proof.
 \end{proof}

 \begin{cor}\label{Nroot}
  If $A\in\F(N)$, $N\in\Spq'$, then
  \[A(\ze_N)=B(\ze_N)=0,\]
  for any spectrum $B\ssq\ZZ_N$ of $A$.
 \end{cor}
 
 \begin{proof}
  By Proposition \ref{primitivereduction}, $A$ and every spectrum $B$ are primitive. By Lemma \ref{primitive}, we get $(A-A)\cap\ZZ_N^{\star}\neq\vn\neq(B-B)\cap\ZZ_N^{\star}$, so by Theorem
  \ref{mainref}\ref{i} we obtain $A(\ze_N)=B(\ze_N)=0$.
 \end{proof}

 Consider $A\in\F(N)$, $N\in\Spq'$, and $B$ a spectrum of $A$.
 By Corollary \ref{Nroot} and Theorem \ref{vansumspos}, we may write
 \begin{equation}\label{ax}
  A(X)\equiv P(X)\Php+Q(X)\Phq\bmod(X^N-1),
 \end{equation}
 for $P(X),Q(X)\in\ZZ_{\geq0}[X]$, and
 \begin{equation}\label{bx}
  B(X)\equiv R(X)\Php+S(X)\Phq\bmod(X^N-1),
 \end{equation} 
 for $R(X),S(X)\in\ZZ_{\geq0}[X]/(X^N-1)$. Next, we will show that Fuglede's conjecture is satisfied if $P$ or $Q$ is identically zero.
 
 \begin{prop}\label{Apq}
  Suppose that $A\in\F(N)$, $N\in\Spq'$. Then $A$ cannot be expressed as a union of $p$- or $q$-cycles exclusively (i.e. $Q\equiv 0$ or $P\equiv 0$, respectively).
 \end{prop}
 
 \begin{proof}
 Let $B$ be a spectrum of $A$, and suppose that $A$ is a union of $p$- or $q$-cycles exclusively.
  Without loss of generality we may assume $Q\equiv 0$, so that 
  \begin{equation}\label{Apcycles}
  A(X)\equiv P(X)\Php\bmod(X^N-1).
  \end{equation} 
  For every $a\in A$ we have $a+\frac{N}{p}\ZZ_N\ssq A$, hence $(A-A)\cap\frac{N}{p}\ZZ_N^{\star}\neq\vn$, implying $B(\ze_p)=0$ and 
  \[\abs{B_{j\bmod p}}=\frac{1}{p}\abs{B},\	\forall j.\]
  Since $P(X)\in\ZZ_{\geq0}[X]$, the coefficients must be either $0$ or $1$, therefore $P(X)$ is the mask
  polynomial of a set $P\ssq\ZZ_N$. Furthermore, \eqref{Apcycles} shows that any element of $P$ is unique $\bmod N/p$, so we may consider this as a subset in $\ZZ_{N/p}$.
  
  Next, we will show that each $B_{j\bmod p}$ is a spectrum for $P$. Since 
  \[B_{j\bmod p}-B_{j\bmod p}\ssq p\ZZ_N\cap (Z(A)\cup\set{0}),\]
  by Theorem \ref{mainref}\ref{i}, it suffices to show that $A(\ze_N^d)=0$ implies $P(\ze_N^d)=0$
  whenever $p\mid d$ and $d\mid N$, i.e. $A(X)$ and $P(X)$ have the same roots, when we restrict to the $\frac{N}{p}$-roots of unity. But this is true, since \eqref{Apcycles} implies
  \[A(X)\equiv pP(X)\bmod (X^{N/p}-1),\]
  therefore, $A(\ze_{N/p}^d)=pP(\ze_{N/p}^d)$, for all $d$. We observe that each spectrum $-j+B_{j\bmod p}$ is a subset of the subgroup of $\ZZ_N$ of order $N/p$, which shows that $P\ssq\ZZ_{N/p}$
  is spectral, since $\abs{P}=\abs{B_{j\bmod p}}$, for all $j$.
  
  By definition of $\Spq'$, $P$ satisfies \ref{t1} and \ref{t2}. We remind that $A(X)$ and $P(X)$ have precisely the same roots, when we restrict to the $\frac{N}{p}$-roots
  of unity. In addition, $A(X)$ vanishes for all other $N$th roots of unity, as $\Php\mid A(X)$, clearly showing that $A(X)$ satisfies \ref{t1} and \ref{t2} as well.
 \end{proof}
 
 One might expect that a similar proof applies when $B$ is a union of $p$- or $q$-cycles exclusively. However, we didn't manage to find such a straightforward proof, but rather a complicated one. The crucial
 ideas needed for this are examined in the next Sections.
 
 

 \bigskip\bigskip
 \section{The absorption-equidistribution property}\label{absorptionequid}
 \bigskip\bigskip
 
 
 \begin{defn}
  Let $A\ssq\ZZ_N$, and let $pd\mid N$, where $p$ is a prime. We call the set $A_{i\bmod d}$ \emph{absorbed} $\bmod pd$, if there is $k$ with $0\leq k\leq p-1$, such that $A_{i\bmod d}=A_{i+kd\bmod pd}$.
  We call it \emph{equidistributed} $\bmod pd$, if $\abs{A_{i+kd\bmod pd}}=\frac{1}{p}\abs{A_{i\bmod d}}$ for every $k$ with $0\leq k\leq p-1$.
 \end{defn}
 
 A fundamental observation coming from the structure of the vanishing sums of roots of unity, i.e. Theorem \ref{vansumspos}, is the fact that $A(\ze_{p^{k+1}})=0$ implies that each $A_{i\bmod p^k}$
 is equidistributed $\bmod p^{k+1}$.
 
 \begin{lemma}\label{abseqd}
  Let $(A,B)$ be a spectral pair in $\ZZ_N$, $N=p^mq^n$, and $d$ a divisor of $\frac{N}{pq}$ such that 
  \[A(\ze_N^d)B(\ze_{pd})\neq0=B(\ze_{pqd}).\]
  Then each subset $B_{i\bmod d}$ is either absorbed or equidistributed $\bmod pd$. Absorption occurs for at least one $i$.
 \end{lemma}
 
 \begin{proof}
  By Theorem \ref{vansumspos}, the multiset $\frac{N}{pqd}\cdot B$ is a disjoint union of $p$- and $q$-cycles. Each such cycle is a subset of a unique $pq$-cycle (i.e. a coset of the subgroup
  $\frac{N}{pq}\ZZ_N$), as $j+\frac{N}{p}\ZZ_N\ssq j+\frac{N}{pq}\ZZ_N$, for example. Therefore, the same conclusion holds for the elements of $\frac{N}{pqd}\cdot B$ (\emph{counting
  multiplicities}) in an arbitrary $pq$-cycle. If $\frac{N}{pqd}\cdot B\cap(j+\frac{N}{pq}\ZZ_N)$ is not supported on a single $p$- or $q$-cycle, then there are $b,b'\in B$ such that
  $\frac{N}{pqd}(b-b')\in\frac{N}{pq}\ZZ_N^{\star}$, by virtue of Proposition \ref{preroothunt}.
	This yields $b-b'\in d\ZZ_N^{\star}$, contradicting the fact that $A(\ze_N^d)\neq0$, using Theorem \ref{mainref}\ref{i}. Now, every element of
  $\frac{N}{pqd}\cdot B_{i\bmod d}$ belongs to the $pq$-cycle $\frac{iN}{pqd}+\frac{N}{pq}\ZZ_N$, therefore it must also be a disjoint union of $p$- and $q$-cycles; in our case, it must be 
  a single $p$- or $q$-cycle with a certain multiplicity. If $\frac{N}{pqd}\cdot B_{i\bmod d}$ is a $p$-cycle with a multiplicity, then this $p$-cycle must be
  \[\frac{iN}{pqd}+k\frac{N}{q}+\frac{N}{p}\ZZ_N=\frac{N}{pqd} B_{i\bmod d},\]
  for some $k$, which clearly shows that $\abs{B_{i+kpd+jqd\bmod pqd}}$ is constant in $j$, and 
  \[B_{i+k'pd+jqd\bmod pqd}=\vn,\]
  for $k'\not\equiv k\bmod q$. This yields the absorption
  \[B_{i+jqd\bmod pd}=B_{i+kpd+jqd\bmod pqd},\]
  for all $j$, which shows that $B_{i\bmod d}$ is equidistributed $\bmod pd$. If on the other hand, $\frac{N}{pqd}\cdot B_{i\bmod d}$ is a $q$-cycle with a multiplicity, then this $q$-cycle must be
  \[\frac{iN}{pqd}+j\frac{N}{p}+\frac{N}{q}\ZZ_N=\frac{N}{pqd} B_{i\bmod d},\]
  for some $j$, which shows that $\abs{B_{i+jqd+kpd\bmod pqd}}$ is constant in $k$, and
  \[B_{i+j'qd+kpd\bmod pqd}=\vn,\]
  for $j'\not\equiv j\bmod p$. This yields
  \[\abs{B_{i\bmod d}}=\abs{B_{i+jqd\bmod pd}}=p\abs{B_{i+jqd+kpd\bmod pqd}},\]
  for this value of $j$ and every $k$, which shows that $B_{i\bmod d}$ is absorbed $\bmod pd$. Absorption $\bmod pd$ must occur for at least one $i$, 
	otherwise we would have $B(\ze_{pd})=0$, contradiction.
 \end{proof}

 \begin{lemma}\label{ascend}
  Let $(A,B)$ be a spectral pair in $\ZZ_N$, $N=p^mq^n$, such that $B_{i\bmod p^k}$ is absorbed $\bmod p^{k+1}$ for every $i$, where $0\leq k<m$. Then, 
  \[\ol{A}(X)\equiv A(X)\Phi_{p^{m-k}}(X^{q^n})\bmod(X^N-1)\]
  is the mask polynomial of a spectral set, whose spectrum has the following mask polynomial:
  \[\ol{B}(X)\equiv B(X)\Phi_{p^{k+1}}(X^{q^n})\bmod(X^N-1).\]
  Moreover, $\ol{B}_{i\bmod p^k}$ is equidistributed $\bmod p^{k+1}$ for every $i$.
 \end{lemma}

 \begin{proof}
  Denote by $T$ the set whose mask polynomial is $\Phi_{p^{m-k}}(X^{q^n})$, so that
  \[T=\set{0,p^{m-k-1}q^n,2p^{m-k-1}q^n,\dotsc,(p-1)p^{m-k-1}q^n}.\]
  Since $B_{i\bmod p^k}$ is absorbed $\bmod p^{k+1}$ for every $i$, we have $B(\ze_{p^{k+1}})\neq0$, therefore by Theorem \ref{mainref}\ref{i} we get
  \[(A-A)\cap\frac{N}{p^{k+1}}\ZZ_N^{\star}=\vn.\]
  Moreover,
  \[T-T=\set{\ell p^{m-k-1}q^n:\abs{\ell}<p}\ssq \set{0}\cup\frac{N}{p^{k+1}}\ZZ_N^{\star},\]
  therefore $(A-A)\cap(T-T)=\set{0}$, hence every element in $A+T$ has a unique representation as $a+t$, with $a\in A$, $t\in T$ (proof is similar to that of Theorem \ref{mainref}\ref{ii}). This shows that
  $\ol{A}(X)$ is a mask polynomial of the set $\ol{A}=A\oplus T$. Furthermore, since each $B_{i\bmod p^k}$ is absorbed $\bmod p^{k+1}$, if $p^k\mid b-b'$ for some $b,b'\in B$, then $p^{k+1}\mid b-b'$ (definition
  of absorption). In other words,
  \begin{equation}\label{BminusB}
  (B-B)\cap (p^k\ZZ_N^{\star}\cup p^kq\ZZ_N^{\star}\cup\dotsb\cup p^kq^n\ZZ_N^{\star})=\vn.
  \end{equation}
  The set
  \begin{equation}\label{extendS}
  S=\set{0,p^kq^n,2p^kq^n,\dotsc,(p-1)p^kq^n}
  \end{equation}
  has mask polynomial $\Phi_{p^{k+1}}(X^{q^n})$ and satisfies
  \begin{equation}\label{SminusS}
  S-S=\set{\ell p^kq^n:\abs{\ell}<p}\ssq \set{0}\cup p^kq^n\ZZ_N^{\star},
  \end{equation}
  therefore $(B-B)\cap(S-S)=\set{0}$ and similarly as before, we deduce that $\ol{B}(X)$ is the mask polynomial of the set $\ol{B}=B\oplus S$. 
	We observe that $\abs{\ol{A}}=\abs{\ol{B}}$.
  
  Next, we will show that $\ol{A}(\ze_{\ord(\be-\be')})=0$ for every $\be,\be'\in\ol{B}$, $\be\neq \be'$. Every $\be\in \ol{B}$ has a unique representation as $b+s$, where $b\in B$ and $s\in S$; we write also
  $\be'=b'+s'$, $b'\in B$, $s'\in S$. By \eqref{SminusS}, $s-s'\in \set{0}\cup p^kq^n\ZZ_N^{\star}$, while $v_p(b-b')\neq k$ by \eqref{BminusB}. If $s=s'$, then $\be-\be'=b-b'$ and
  \[\ol{A}(\ze_{\ord(b-b')})=A(\ze_{\ord(b-b')})\Phi_{p^{m-k}}(\ze_{\ord(b-b')}^{q^n})=0,\]
  by Corollary \ref{specord}, since $B$ is a spectrum of $A$. So, assume $s\neq s'$, so that $s-s'\in p^kq^n\ZZ_N^{\star}$. Since $v_p(b-b')\neq k$, we will either have $p^{k+1}\mid b-b'$ or $p^k\nmid b-b'$.
  In the former case, $v_p(\be-\be')=k$, yielding $\ord(\be-\be')=p^{m-k}q^{r}$, where $0\leq r\leq n$. Hence,
  \[\ol{A}(\ze_{\ord(\be-\be')})=A(\ze_{\ord(\be-\be')})\Phi_{p^{m-k}}(\ze_{p^{m-k}q^r}^{q^n})=A(\ze_{\ord(\be-\be')})\Phi_{p^{m-k}}(\ze_{p^{m-k}}^{q^{n-r}})=0.\]
  In the latter case, $\ord(s-s')=p^{m-k}\mid \ord(b-b')$, therefore $\ord(b-b')=\ord(b-b'+s-s')=\ord(\be-\be')$, hence again Corollary \ref{specord} gives $\ol{A}(\ze_{\ord(\be-\be')})=0$. One more application
  of Corollary \ref{specord} gives the desired conclusion, as $\ol{A}(\ze_{\ord(\be-\be')})=0$, for every $\be,\be'\in \ol{B}$, $\be\neq\be'$. The second part of the Lemma follows from the fact that
  $\ol{B}(\ze_{p^{k+1}})=0$.
  \end{proof}

 \begin{cor}\label{multascend}
  Let $(A,B)$ be a spectral pair in $\ZZ_N$, $N=p^mq^n$. Let $K_p(B)\ssq[0,m-1]$ be the set of all $k$ such that $B_{i\bmod p^k}$ is absorbed $\bmod p^{k+1}$ for all $i$. Then
  \[\ol{A}(X)\equiv A(X)\prod_{k\in K_p(B)}\Phi_{p^{m-k}}(X^{q^n})\bmod(X^N-1)\]
  is the mask polynomial of a spectral set, whose spectrum has the following mask polynomial:
  \[\ol{B}(X)\equiv B(X)\prod_{k\in K_p(B)}\Phi_{p^{k+1}}(X^{q^n})\bmod(X^N-1).\]
  Moreover, for every $k$ with $p^{k+1}\mid N$, there is $i$ such that $\ol{B}_{i\bmod p^k}$ is not absorbed $\bmod p^{k+1}$ (i.e. the set $K_p(\ol{B})$ is empty).
 \end{cor}
 
 \begin{proof}
  Induction on $\abs{K_p(B)}$; the case $\abs{K_p(B)}=1$ is Lemma \ref{ascend}, and the proof for the induction step is almost identical.
  For the second part, we proceed by contradiction; suppose that $K_p(\ol{B})\neq\vn$. By the second part of Lemma \ref{ascend}, we definitely have $K_p(B)\cap K_p(\ol{B})=\vn$, as for every $k\in K_p(B)$
  it holds $\ol{B}(\ze_{p^{k+1}})=0$, hence each $\ol{B}_{i\bmod p^k}$ is equidistributed $\bmod p^{k+1}$ (and not absorbed).
  
  Consider next some arbitrary $\ell\notin K_p(B)$, $0\leq \ell\leq m$. By definition, there is $i$ such that $B_{i\bmod p^{\ell}}$ is not absorbed $\bmod p^{\ell+1}$. Then, as $B\ssq\ol{B}$, it is obvious
  that $\ol{B}_{i\bmod p^k}$ is not absorbed $\bmod p^{k+1}$.
 \end{proof}

 \begin{defn}
  We call a set $A\ssq\ZZ_N$ \emph{absorption-free at $p^k$}, if for a prime $p$ and $k\in\NN$ with $p^{k+1}\mid N$, there is a class $A_{i\bmod p^k}$ that is not absorbed $\bmod p^{k+1}$ (i.e.
  $k\notin K_p(A)$).
  We call $A$ \emph{absorption-free}, if it is absorption-free at every prime power $p^k$, such that $p^{k+1}\mid N$ (i.e. $K_p(A)=\vn$, for every prime $p\mid N$).
 \end{defn}
 
 \begin{lemma}\label{absfreeredux}
  Let $(A,B)$ be a spectral pair in $\ZZ_N$, $N=p^mq^n$. Then, there are $S,T\ssq\ZZ_N$ such that $(A\oplus S,B\oplus T)$ is a spectral pair in $\ZZ_N$, with the additional property that $A\oplus S$ and $B\oplus T$ are
  absorption-free.
 \end{lemma}
 
 \begin{proof}
  If both $A$ and $B$ are absorption-free, there is nothing to prove, as we can just take $S=T=\set{0}$. So, suppose first that $B$ is not absorption-free. The process described in Lemma \ref{ascend} and subsequently
  in Corollary \ref{multascend} ``kills'' all absorption, in the following sense: if $B_{i\bmod p^k}$ is absorbed $\bmod p^{k+1}$ for all $i$, then $\ol{B}_{i\bmod p^k}$ is 
  equidistributed $\bmod p^{k+1}$ for all $i$, since $\ol{B}(\ze_{p^{k+1}})=0$, and thus $\ol{B}$ is absorption-free at $p^k$. Moreover, for every $k\in[0,m-1]$, there is $i$ such that $\ol{B}_{i\bmod p^k}$ is not
  absorbed $\bmod p^{k+1}$. With the goal to eliminate all absorption from $A$ and $B$,
  we successively multiply $B(X)$ by polynomials of the form $\Phi_{p^{k+1}}(X^{q^n})$ or $\Phi_{q^{\ell+1}}(X^{p^m})$, whose product is a mask polynomial of a set $S$, 
  and respectively $A(X)$ by polynomials of the form $\Phi_{p^{m-k}}(X^{q^n})$ or $\Phi_{q^{n-\ell}}(X^{p^m})$, whose product is a mask polynomial of a set $T$.
  We thus obtain a pair of absorption-free 
  sets $(A\oplus S,B\oplus T)$, which is also spectral, by virtue of Lemma \ref{ascend} or Corollary \ref{multascend}.
 \end{proof}

 \begin{cor}\label{absfreetile}
  Suppose that every absorption-free spectral subset of $\ZZ_N$ tiles $\ZZ_N$. Then, every spectral subset of $\ZZ_N$ tiles $\ZZ_N$.
 \end{cor}
 
 \begin{proof}
  Let $(A,B)$ be a spectral pair in $\ZZ_N$. If $A$ is absorption-free, then it tiles $\ZZ_N$ by hypothesis. If not, there are $S,T\ssq\ZZ_N$ by Lemma \ref{absfreeredux}, such that $(A\oplus T, B\oplus S)$ is
  an absorption-free spectral pair. By hypothesis, this shows the existence of $R\ssq\ZZ_N$, such that $(A\oplus T)\oplus R=\ZZ_N$, or alternatively, $A\oplus(T\oplus R)=\ZZ_N$, that is, $A$ tiles $\ZZ_N$.
 \end{proof}

 All the above motivate the following definition.
 
 \begin{defn}
  Let $N\in\NN$. Define the following subset of $\F(N)$
  \[\F_{\max} (N)=\set{A\ssq\ZZ_N:A\in \F(N), \text{ and }A\text{ has maximal size}}.\]
 \end{defn}
 
 Lemma \ref{absfreeredux} and Corollary \ref{absfreetile} imply that every $A\in\F_{\max}(N)$ is absorption-free, as well as every spectrum $B$ of $A$. Hence, Corollary \ref{absfreetile} implies
 \begin{equation}\label{absfreespec}
  N\in\Spq \Longleftrightarrow \F_{\max}(N)\neq\vn.
 \end{equation}

 
 

 \bigskip\bigskip
 
 \section{Symmetry of root patterns}
 \bigskip\bigskip
 
 We consider again $N\in\Spq'$, $A\in \F(N)$, and $B$ a spectrum of $A$. We recall Propositions \ref{primitivereduction},
\ref{Apq}, and Corollary \ref{Nroot}, which imply that $A$ and $B$ are both primitive, they satisfy \eqref{ax} and
\eqref{bx}, such that the polynomials $P,Q$ in \eqref{ax} are nonzero. Therefore, $A$ contains a coset from each of
 the subgroups $\frac{N}{p}\ZZ_N$ and $\frac{N}{q}\ZZ_N$, so using Theorem \ref{mainref}\ref{i} we obtain
\[\frac{N}{p}\ZZ_N\cup\frac{N}{q}\ZZ_N\ssq A-A\ssq\set{0}\cup Z(B),\]
and then Corollary \ref{specord} gives
 \begin{equation}\label{Bpqroots}
  B(\ze_p)=B(\ze_q)=0,
 \end{equation}
for any spectrum $B$ of $A$. If in addition $A\in\Fmax(N)$ holds, then $A$ and every spectrum $B$ are absorption-free by Lemma \ref{absfreeredux}. 
 
 \begin{defn}\label{m0n0}
  For $A\ssq\ZZ_N$ we define 
  \[S_A^N(p):=\set{\log_p s:s\in S_A^N, p\mid s}=\set{x\in[1,m]:A(\ze_{p^x})=0},\]
  and the definition of $S_A^N(q)$ is similar. Also, denote by $m_p(A)$ the maximal exponent $x\leq m$ such that $A(\ze_N^{p^x})\neq0$, and similarly define $m_q(A)$ to be the maximal exponent $y\leq n$ 
  such that $A(\ze_N^{q^y})\neq0$; if such exponents do not exist, we put $m_p(A)=-1$, $m_q(A)=-1$, respectively. 
    Lastly, define the following set
  \[
   R_A^N(p):=\set{x\in[0,m]:A(\ze_N^{p^x})=0},
  \]
  and similarly $R_A^N(q)$. The sets thus defined will be called \emph{root patterns} of $A$.
 \end{defn}

We remark that if the sets 
below are nonempty, then
\[m_p(A)=\max\bra{[0,m]\sm R_A^N(p)}, \;\;\;\; m_q(A)=\max\bra{[0,m]\sm R_A^N(q)}.\]
 For $A\in\F(N)$, the above sets are always nonempty, i.e. $m_p(A),n_p(A)\geq0$. This is a consequence of Ma's Lemma for the cyclic case \cite{Ma} (see also Lemma 1.5.1 \cite{SchmidtFDM}, or 
 Corollary 1.2.14 \cite{Pott}), using the polynomial notation.

\begin{lemma}\label{ma}
 Suppose that $A(X)\in\ZZ[X]$, and let $\ZZ_N$ be a cyclic group such that $p^m\pdiv N$, for $p$ prime. If $A(\ze_d)=0$, for every $p^m\mid d\mid N$, then
 \[A(X)\equiv P(X)\Php\bmod(x^N-1).\]
 If the coefficients of $A$ are nonnegative, then $P$ can be taken with nonnegative coefficients as well. In particular, if $A\subseteq\ZZ_N$ satisfies $A(\ze_d)=0$, for every $p^m\mid d\mid N$, then
 $A$ is a union of $p$-cycles.
\end{lemma}

Therefore, if $A\ssq\ZZ_N$ satisfies $m_p(A)=-1$ (resp. $m_q(A)=-1$), then it is a union of $q$-cycles (resp. $p$-cycles), so if $N\in\Spq'$ and $A$ spectral, then $A$ is also a tile by
Proposition \ref{Apq}.
 
 \begin{lemma}\label{rootpatterns}
  Let $N\in\Spq'$, $A\in\F(N)$ and $B$ any spectrum of $A$. Then, 
  we have the following relation between the root patterns of $A$ and $B$:
  \begin{eqnarray}\label{ABpsym}
  (S_B^N(p)-1)\cap[0,m_p(A)] &=& R_A^N(p)\cap[0,m_p(A)]\\ \label{ABqsym}
  (S_B^N(q)-1)\cap[0,m_q(A)] &=& R_A^N(q)\cap[0,m_q(A)]. 
  \end{eqnarray}
	Equations \eqref{ABpsym} and \eqref{ABqsym} are also 
  equivalent to the statements
  \begin{eqnarray}\label{ABpsymalt}
   \text{If }x\leq m_p(A), \text{ then it holds } A(\ze_N^{p^x})=0\Longleftrightarrow B(\ze_{p^{x+1}})=0\\ \label{ABqsymalt}
   \text{If }x\leq m_q(A), \text{ then it holds } A(\ze_N^{q^x})=0\Longleftrightarrow B(\ze_{q^{x+1}})=0. 
  \end{eqnarray}
 \end{lemma}
 
 \begin{proof}
Since $A\in\F(N)$ and $N\in\Spq'$, we obtain that
  $A$ cannot be expressed as a union of $p$- or $q$-cycles exclusively, by virtue of Theorem \ref{vansumspos}, Corollary \ref{Nroot} and
	Proposition \ref{Apq}. Thus,
	$m_p(A), m_q(A)\geq0$. Let $x< m_p(A)$ and $A(\ze_N^{p^x})=0$. Then, $p^x\cdot A$ is a union of
  $p$- and $q$-cycles; not exclusively of $q$-cycles by Theorem \ref{vansumspos} (second part), because $A(\ze_N^{p^{m_p(A)}})\neq0$. 
	So, there are $a,a'\in A$, such that 
  \[p^x(a-a')\equiv\frac{N}{p}\bmod N,\]
  or equivalently,
  \[a-a'\equiv\frac{N}{p^{x+1}}\bmod \frac{N}{p^x},\]
  which in turn implies $a-a'\in\frac{N}{p^{x+1}}\ZZ_N^{\star}$, whence $B(\ze_{p^{x+1}})=0$ by Theorem \ref{mainref}. This proves one
	direction of \eqref{ABpsymalt}.
  
  Next, suppose that the reverse implication of \eqref{ABpsymalt} does not hold, and let $y$ be the smallest exponent such that 
	$A(\ze_N^{p^y})\neq0=B(\ze_{p^{y+1}})$. For every $x\leq y$ and every $i$ we have the obvious partition
  \[B_{i\bmod p^x}=\bigcup_{k=1}^p B_{i+kp^x\bmod p^{x+1}}.\]
  If $A(\ze_N^{p^x})=0$, so that $B(\ze_{p^{x+1}})=0$, the above partition is an equidistribution $\bmod p^{x+1}$, that is
  \[\abs{B_{i\bmod p^x}}=p\abs{B_{i+kp^x\bmod p^{x+1}}},\]
  for every $k$. If $A(\ze_N^{p^x})\neq0\neq B(\ze_{p^{x+1}})$, we certainly do not have equidistribution for at least one choice of $j$; if we do not have absorption $\bmod p^{x+1}$, i.e.
  $B_{i\bmod p^x}=B_{i+kp^x\bmod p^{x+1}}$, then there are $k,k'$ with $k\not\equiv k'\bmod p$, such that $B_{i+kp^x\bmod p^{x+1}}$ and $B_{i+k'p^x\bmod p^{x+1}}$ are nonempty. Let
  \[b\in B_{i+kp^x\bmod p^{x+1}}, b'\in B_{i+k'p^x\bmod p^{x+1}},\]
  be arbitrary.
  It holds $p^x\pdiv b-b'$ by assumption, and $b-b'\notin p^x\ZZ_N^{\star}$ due to $A(\ze_N^{p^x})\neq0$. This shows $q\mid b-b'$, therefore all elements of $B_{i+kp^x\bmod p^{x+1}}$ and
  $B_{i+k'p^x\bmod p^{x+1}}$ have the same residue $\bmod q$. The same holds for all other elements of $B_{i\bmod p^x}$, establishing
  \[B_{i\bmod p^x}\ssq B_{j\bmod q}\]
  for some $j$. Now, for each $i\in[0,p^y-1]$ define $x(i)$ to be the smallest exponent $x$ such that $B_{i\bmod p^x}$ is neither absorbed nor equidistributed $\bmod p^{x+1}$, and then define $j(i)$ such that
  \begin{equation}\label{pfitsq}
  B_{i\bmod p^{x(i)}}\ssq B_{j(i)\bmod q}.
  \end{equation}
  If for every $x\leq y$, $B_{i\bmod p^x}$ is either absorbed or equidistributed $\bmod p^{x+1}$, define $x(i)=y$; since $B_{i\bmod p^y}$ is equidistributed $\bmod p^{y+1}$ and $A(\ze_N^{p^y})\neq0$, we conclude
  again that there is again some $j(i)$ such that \eqref{pfitsq} still holds. On $[0,p^y-1]$ we define an equivalence relation: it holds $i\sim i'$ if and only if $i\equiv i'\bmod p^{x(i)}$ (it is not hard
  to show that this satisfies the properties of an equivalence relation, as $i\equiv i'\bmod p^{x(i)}$ implies $x(i)=x(i')$).
   Let $I\ssq[0,p^y-1]$ be a complete system of representatives; it holds
  \[B=\bigsqcup_{i\in I}B_{i\bmod p^{x(i)}}.\]
  The important fact about the above partition, is that each $B_{i\bmod p^{x(i)}}$ has cardinality $\frac{\abs{B}}{p^{k(i)}}$, for some positive
	integer $k(i)$. Indeed, as for every $x<x(i)$ we have either $\abs{B_{i\bmod p^x}}=\abs{B_{i\bmod p^{x+1}}}$ (absorption) or
	$\abs{B_{i\bmod p^x}}=p\abs{B_{i\bmod p^{x+1}}}$ (equidistribution).
	Then, by \eqref{pfitsq} we have
  \[B_{0\bmod q}=\bigsqcup_{i\in I, j(i)=0}B_{i\bmod p^{x(i)}},\]
  and taking cardinalities on both sides we obtain
  \[\frac{\abs{B}}{q}=\sum_{i\in I, j(i)=0}\frac{\abs{B}}{p^{k(i)}},\]
  which then leads to an equation of the form $1/q=s/p^t$, which has no solutions in integers, contradicting the fact that $A(\ze_N^{p^y})\neq0=B(\ze_{p^{y+1}})$. This completes the proof.
 \end{proof}

 \begin{prop}\label{maxpqroots}
  Let $N\in\Spq'$, $A\in\F(N)$, and $B$ be any spectrum of $A$. Then, 
  \[A(\ze_{q^n})=A(\ze_{p^m})=0\]
  and
  \[B(\ze_{p^{m_p(A)+1}q})=B(\ze_{pq^{m_q(A)+1}})=0.\]
 \end{prop}
 
 \begin{proof}
  We will show first that $A(\ze_{q^n})=A(\ze_{p^m})=0$. Suppose that it doesn't hold, say $A(\ze_{q^n})\neq0$, so that $m_p(A)=m$. By \eqref{ABpsym} we get
  \[S_B^N(p)-1 = R_A^N(p).\]
  Using the same argument as in the proof of Lemma \ref{rootpatterns}, replacing $y$ by $m$ we get
  \[B=\bigsqcup_{i\in I}B_{i\bmod p^{x(i)}},\]
  where each $B_{i\bmod p^{x(i)}}$ has cardinality $\frac{\abs{B}}{p^{k(i)}}$ and is contained in $B_{j(i)\bmod q}$; this is also true if $x(i)=m$, as $A(\ze_N^{p^m})\neq0$. We establish a contradiction 
  in the same way.
  
  This shows $m_p(A)<m$ and $m_q(A)<n$. Since $B(\ze_{p^{m_p(A)+1}})\neq0$ by \eqref{ABpsymalt}, the polynomial
  \[\ol{A}(X)\equiv A(X)\Phi_p(X^{p^{m-m_p(A)-1}q^n})\bmod(X^N-1)\]
  is the mask polynomial of a set, say $\ol{A}$; this follows from the fact that $(A-A)\cap\frac{N}{p^{m_p(A)+1}}\ZZ_N^{\star}=\vn$, which implies that $A-A$ and the difference set whose mask polynomial
  is $\Phi_p(X^{p^{m-m_p(A)-1}q^n})$ intersect trivially. Consider now the multiset $p^{m_p(A)}\cdot A$. By definition, its mask polynomial does not vanish on $\ze_N$, while on the other hand, $p^{m_p(A)+1}\cdot A$
  is a disjoint union of $q$-cycles by Lemma \ref{ma}. By Lemma \ref{roothunt} and Theorem \ref{mainref}\ref{i}, we obtain $B(\ze_{p^{m_p(A)+1}q})=0$, and similarly $B(\ze_{pq^{m_q(A)+1}})=0$, as desired.  
 \end{proof}
 
 Using the above notation, condition \ref{t1} can be rewritten as
 \[\abs{A}=p^{\abs{S_A^N(p)}}q^{\abs{S_A^N(q)}}.\]
 At any rate, regardless whether $A$ is spectral or not, it holds
 \[p^{\abs{S_A^N(p)}}q^{\abs{S_A^N(q)}}\mid \abs{A}.\]
 We restrict now to $A\in \Fmax(N)$, $N\in\Spq'$, due to \eqref{absfreespec}.
 In this case, we will show that \ref{t1} fails for $B$ when the multiset $\frac{N}{p^{m_p(A)+1}q}\cdot B$ is not a union of $q$-cycles only; in particular, the power of $p$ is higher 
 than it is supposed to be.

 \begin{lemma}\label{t1fail}
  Let $N\in\Spq'$, $A\in\Fmax(N)$, and $B$ any spectrum of $A$, so that $m_p(A), m_q(A)\geq0$. Consider the set
  \[U_B^N(p)=\set{x: B(\ze_{p^x})\neq0=B(\ze_{p^xq})}.\]
  Then,
  \[p^{\abs{S_B^N(p)}+\abs{U_B^N(p)}}\mid \abs{B}.\]
  Moreover, for every $j$, every $x\in [1,m_p(A)+1]\sm S_B^N(p)$, and every $y\in [1,m_q(A)+1]\sm S_B^N(q)$, each of the sets
  \[\bra{\frac{N}{p^xq}B}_{j\bmod\frac{N}{pq}}=\frac{N}{p^xq}B\cap(j+\frac{N}{pq}\ZZ_N), \;\;\;\;	\bra{\frac{N}{pq^y}B}_{j\bmod\frac{N}{pq}}=\frac{N}{pq^y}B\cap(j+\frac{N}{pq}\ZZ_N),\]
  is supported either on a $p$- or a $q$-cycle.
 \end{lemma}
 
 \begin{proof}
  We first note, that $x=m_p(A)+1$ satisfies $B(\ze_{p^x})\neq0=B(\ze_{p^xq})$ by Lemma \ref{rootpatterns} and Proposition \ref{maxpqroots}, therefore $U_B^N(p)\neq\vn$. 
  By \eqref{Bpqroots} we have $B(\ze_p)=0$, or equivalently, $\abs{B_{i\bmod p}}=\frac{1}{p}\abs{B}$ for all $i$. At every step, we will estimate
  \[\min\set{i:v_p(\abs{B_{i\bmod p^x}})},\]
  or in other words, how much the power of $p$ drops from $\abs{B}$ down to one of the subsets $\abs{B_{i\bmod p^x}}$. If $x\in S_B^N(p)$, then each $B_{i\bmod p^{x-1}}$ is equidistributed $\bmod p^x$, therefore
  \[v_p(\abs{B_{i\bmod p^{x-1}}})-1=v_p(\abs{B_{i\bmod p^x}}),\]
  for all $i$. If $x\notin S_B^N(p)$, we have in general
  \[v_p(\abs{B_{i\bmod p^{x-1}}})\geq \min\set{v_p(\abs{B_{i+kp^{x-1}\bmod p^x}}):0\leq k\leq p-1}.\]
  Next, consider the minimum element of $U_B^N(p)$, say $x$, and suppose that $B_{i\bmod p^{x-1}}$ is not absorbed $\bmod p^x$; such an $i$ exists due to hypothesis and
  \eqref{absfreespec}, since $B$ is absorption-free. Each nonempty class $B_{i+kp^{x-1}\bmod p^x}$ would then be absorbed $\bmod p^xq$, 
  otherwise $\frac{N}{p^xq}B\cap(j+\frac{N}{pq}\ZZ_N)$
  would neither be supported on a $p$-cycle, nor on a $q$-cycle, for some $j$. For the same reason, each element in $B_{i+kp^{x-1}\bmod p^x}$ must have the same residue $\bmod q$, for all $k$, 
  showing that $B_{i\bmod p^{x-1}}$ is absorbed $\bmod p^{x-1}q$; furthermore, since $B(\ze_{p^xq})=0$, $B_{i\bmod p^{x-1}}$ and $B_{i\bmod p^{x-1}q}$ are equidistributed $\bmod p^x$ 
  and $\bmod p^xq$, respectively. Therefore,
  \[v_p(\abs{B_{i\bmod p^{x-1}}})-1=v_p(\abs{B_{i\bmod p^x}}).\]
  
  
  Next, let $x$ be the minimum element not belonging to $S_B^N(p)$, that is
  \[B(\ze_p)=\dotsb=B(\ze_{p^{x-1}})=0\neq B(\ze_{p^x}).\]
  Suppose that $B_{i\bmod p^{x-1}}$ is not absorbed $\bmod p^x$; such an $i$ exists due to hypothesis and \eqref{absfreespec}. Let then $B_{i+kp^{x-1}\bmod p^x}$ and
  $B_{i+k'p^{x-1}\bmod p^x}$ be nonempty with $p\nmid k-k'$. If there are $b\in B_{i+kp^{x-1}\bmod p^x}$ and $b'\in B_{i+k'p^{x-1}\bmod p^x}$ with $q\nmid b-b'$, then $b-b'\in p^{x-1}\ZZ_N^{\star}$,
  yielding $A(\ze_N^{p^{x-1}})=0$ by Theorem \ref{mainref}\ref{i}, and then $B(\ze_{p^x})=0$ by Lemma \ref{rootpatterns}, contradicting the definition of $x$. Thus, we must have $q\mid b-b'$ for 
  every such selection of $b,b'$, implying eventually that every element of $B_{i\bmod p^{x-1}}$ has the same residue $\bmod q$.
  
  For every $y\geq x$ we choose an integer $i_y\equiv i\bmod p^{x-1}$ as follows (we denote $i_j=i$ for $1\leq j\leq x-1$): 
  \begin{itemize}
   \item if $y\notin S_B^N(p)\cup U_B^N(p)$, choose $i_y$ satisfying 
   \[v_p(\abs{B_{i_{y-1}\bmod p^{y-1}}})\geq v_p(\abs{B_{i_y\bmod p^y}}).\]
   \item if $y\in S_B^N(p)\cup U_B^N(p)$, then $i_y=i_{y-1}$.
  \end{itemize}
  We have $\abs{B_{i\bmod p^{x-1}}}=\frac{1}{p^{x-1}}\abs{B}$. If $y\in S_B^N(p)$, then $B_{i_{y-1}\bmod p^{y-1}}$ is equidistributed $\bmod p^y$, so
  \begin{equation}\label{pdrop}
  \abs{B_{i_y\bmod p^y}}=\frac{1}{p}\abs{B_{i_{y-1}\bmod p^{y-1}}}.
  \end{equation}
  If $y\in U_B^N(p)$, we have $B(\ze_{p^yq})=0$. Therefore, the multiset $\frac{N}{p^yq}\cdot B$ is a disjoint union of $p$- and $q$-cycles by Theorem \ref{vansumspos}. 
  Every such cycle, consists of elements $p^yqb$ for $b\in B$
  that have the same residue $\bmod p^{y-1}$; indeed, if $\frac{N}{p^yq}(b-b')\in\frac{N}{p}\ZZ_N$, then $b-b'\equiv kp^{y-1}q\bmod p^yq$ for some $k$, and if
  $\frac{N}{p^yq}(b-b')\in\frac{N}{q}\ZZ_N$, then $b-b'\equiv kp^y\bmod p^yq$ for some $k$, proving our claim. This shows that $\frac{N}{p^yq}\cdot B_{i_{y-1}\bmod p^{y-1}}$ is a disjoint union
  of $p$- and $q$-cycles; however, a $q$-cycle in $\frac{N}{p^yq}\cdot B_{i_{y-1}\bmod p^{y-1}}\ssq\frac{N}{p^yq}B$ consists of elements having all possible residues $\bmod q$, a contradiction, 
  since every element of $B_{i\bmod p^{x-1}}$ has the same residue $\bmod q$. This shows
  that $\frac{N}{p^yq}\cdot B_{i_{y-1}\bmod p^{y-1}}$ is a union of $p$-cycles, yielding the equidistribution of $B_{i_{y-1}\bmod p^{y-1}}$ modulo $p^y$. Thus, \eqref{pdrop} holds in this case
  as well.
  
  Summarizing, we obtain the following: if $y\in S_B^N(p)\cup U_B^N(p)$, then \eqref{pdrop} is valid, otherwise 
  \[v_p(\abs{B_{i_{y-1}\bmod p^{y-1}}})\geq v_p(\abs{B_{i_y\bmod p^y}})\]
  holds. This obviously yields that
  \[0\leq v_p(\abs{B_{i_m\bmod p^m}})\leq v_p(\abs{B})-\abs{S_B^N(p)}-\abs{U_B^N(p)},\]
  proving the first part.
  
  For the second part of the Lemma, we proceed by contradiction. If there is some $j$ and $x\in U_B^N(p)$ with $x\leq m_p(A)+1$ 
  such that $\frac{N}{p^xq}B\cap(j+\frac{N}{pq}\ZZ_N)$ is supported neither on a $p$- nor on a $q$-cycle, then
  by Proposition \ref{preroothunt} we get that 
  \[\bra{\frac{N}{p^xq}B-\frac{N}{p^xq}B}\cap\frac{N}{pq}\ZZ_N^{\star}\neq\vn.\]
  This shows the existence of $b,b'\in B$ such that 
  \[\frac{N}{p^xq}(b-b')\in\frac{N}{pq}\ZZ_N^{\star},\]
  or equivalently,
  \[\ord(b-b')=pq\frac{N}{p^xq}=\frac{N}{p^{x-1}},\]
  hence by Theorem \ref{mainref}\ref{i} we obtain 
  \[A(\ze_{N/p^{x-1}})=A(\ze_N^{p^{x-1}})=0,\]
  contradicting the definition of $x$. This completes the proof.
 \end{proof}

 \begin{defn}
  Let $A\ssq\ZZ_N$. We define
  \begin{enumerate}[{\bf(wT1)},leftmargin=*]
   \item $p^{\abs{S_A^N(p)}}\pdiv \abs{A}$, for some $p\mid \abs{A}$. \label{wt1}
  \end{enumerate}
 \end{defn}
 
Obviously, when \ref{t1} holds for $A$, then \ref{wt1} holds for \emph{every} $p\mid A$.
 An immediate consequence of this definition and Lemma \ref{t1fail} is the following.

 \begin{cor}\label{wt1tile}
  Let $N\in\Spq'$, $A\in\Fmax(N)$. Then, no spectrum $B$ of $A$ satisfies \ref{wt1}. In particular, every spectrum $B$ of $A$ satisfies $B\in\Fmax(N)$.
 \end{cor}
 
 \begin{proof}
  Obvious by Lemma \ref{t1fail}, since the sets $U_B^N(p)$ and $U_B^N(q)$ are nonempty. 
 \end{proof}

 \begin{cor}\label{Bpq}
  Let $N\in\Spq'$, $A\in\Fmax(N)$. Then no spectrum $B$ of $A$ is a disjoint union either of $p$- or $q$-cycles exclusively. 
 \end{cor}
 
 \begin{proof}
 If $B$ was a union of (say) $p$-cycles only, then by Proposition \ref{Apq}, $B$ satisfies \ref{t1} and \ref{t2}. On the other hand, Corollary \ref{wt1tile} gives us that \ref{wt1} fails
 for $B$, a contradiction, as \ref{t1} implies \ref{wt1}.
 \end{proof}
 
 This Corollary complements Proposition \ref{Apq}. We summarize below all the reductions made so far:
 
 \begin{thm}\label{summary}
  Let $N\in\Spq'$, $A\in\Fmax(N)$. Then the following hold:
  \begin{enumerate}
  \item $A$ and any spectrum $B$ are primitive.
   \item $A$, as well as any spectrum $B$, can be expressed as disjoint unions of $p$- and $q$-cycles; not exclusively of $p$- or $q$-cycles (i.e. none of the polynomials $P,Q,R,S$ in both \eqref{ax}
   and \eqref{bx} can be identically zero). \label{summary2}
   \item If $N=p^mq^n$, it holds
   \begin{equation}\label{roots1}
   \begin{split}
    A(\ze_N)=A(\ze_p)=A(\ze_q)=A(\ze_{p^m})=A(\ze_{q^n})=& \\
    =B(\ze_N)=B(\ze_p)=B(\ze_q)=B(\ze_{p^m})=B(\ze_{q^n})=& 0
    \end{split}
   \end{equation}
  for any spectrum $B$ of $A$. This also implies that both $A$ and $B$ are equidistributed both $\bmod p$ and $\bmod q$.
  \item $A$ as well as any spectrum $B$ are absorption-free. 
	\item Every spectrum $B$ of $A$ satisfies $B\in\Fmax(N)$.
  \end{enumerate}
 \end{thm}


 We have established enough tools to extend the results of \cite{KMSV20}, related to Fuglede's conjecture.
 
 \begin{cor}\label{pmq2}
  Let $N=p^mq^n$, $n\leq3$, and $(A,B)$ be a spectral pair in $\ZZ_N$. Then $A$ tiles $\ZZ_N$.
 \end{cor}
 
 \begin{proof}
 Suppose on the contrary that there is such a $N$ with $\F(N)\neq\vn$, $A\in\Fmax(N)$ and $B$ a spectrum. Corollary \ref{mkshort} implies that $n\geq2$.
 By Theorem \ref{summary}, we obtain $A(\ze_q)=A(\ze_{q^n})=0$, so if $n=2$
 we have $q^n\mid\abs{A}$, so by Proposition \ref{maxpower} we get that $A$ tiles, contradiction. If $n=3$, we have $B(\ze_q)=B(\ze_{q^3})=0$ again by Theorem \ref{summary}, so that
 $\abs{S_B^N(q)}\geq2$. Since $U_B^N(q)$ is nonempty, by Lemma \ref{t1fail} we obtain $q^3\mid\abs{B}$, and then Proposition \ref{maxpower} gives us again that $A$ tiles $\ZZ_N$,
 a contradiction. Thus, $\Spq$ cannot have elements of the form $p^mq^n$ with $n\leq3$, proving the desired result.
 \end{proof}
 
 
 \begin{prop}\label{basict2}
  Let $N=p^mq^n\in\Spq'$, $A\in\Fmax(N)$, and $B$ any spectrum of $A$. Then,
  \[A(\ze_{p^mq})=A(\ze_{pq^n})=B(\ze_{p^mq})=B(\ze_{pq^n})=A(\ze_{pq})=B(\ze_{pq})=0.\]
 \end{prop}
 
 \begin{proof}
  By Theorem \ref{summary}, we also have $B\in\Fmax(N)$, so we may use Lemma \ref{rootpatterns} with the roles of $A$ and $B$ reversed. If 
  $A(\ze_{pq^n})=A(\ze_N^{p^{m-1}})\neq0=A(\ze_N^{p^m})=A(\ze_{q^n})$, we would have $m_p(A)=m-1$, 
	so by virtue of Lemma \ref{rootpatterns} we would have
  $B(\ze_{p^m})\neq0$, contradicting equations \ref{ABpsym}. Reversing the roles of $p$, $q$, and $A$, $B$, we obtain with the same reasoning
  \[A(\ze_{p^mq})=A(\ze_{pq^n})=B(\ze_{p^mq})=B(\ze_{pq^n})=0.\]
  Next, consider the minimal element of $U_B^N(p)$, say $x$, so that $B(\ze_{p^x})\neq0=B(\ze_{p^xq})$. By Proposition \ref{maxpqroots} we have
  $x\leq m_p(A)+1$, hence by Lemma \ref{rootpatterns} we get $A(\ze_N^{p^{x-1}})\neq0$. Applying Lemma \ref{abseqd} for $d=p^{x-1}$, we get that each $B_{i\bmod p^{x-1}}$ is either absorbed
  or equidistributed $\bmod p^x$. Since $B(\ze_{p^x})\neq0$, there is some $i$ such that $B_{i\bmod p^{x-1}}$ is absorbed $\bmod p^x$; without loss of generality we put $B_{i\bmod p^{x-1}}=
  B_{i\bmod p^x}$. This also shows that $B_{i\bmod p^xq}$ is supported on a $q$-cycle; in fact, $B_{i\bmod p^x}$ must be equidistributed $\bmod p^xq$, hence
  \begin{equation}\label{edmodpxq}
  \abs{B_{i+kp^x\bmod p^xq}}=\frac{1}{q}\abs{B_{i\bmod p^x}}, \;\;\;\; 0\leq k\leq q-1. 
  \end{equation}
  For $y\leq x-1$ we then claim the following: $B_{i\bmod p^y}$
  is either absorbed or equidistributed $\bmod p^{y+1}$. Indeed, if $y+1\in S_B^N(p)$, then $B(\ze_{p^{y+1}})=0$ and $B_{i\bmod p^y}$ is equidistributed $\bmod p^{y+1}$. If $B(\ze_{p^{y+1}})\neq0$, then
  $A(\ze_N^{p^y})\neq0$ by Lemma \ref{rootpatterns}. If $B_{i\bmod p^y}$ is not absorbed $\bmod p^{y+1}$, then there are $k\not \equiv k'\bmod p$ such that 
  \[B_{i+kp^y\bmod p^{y+1}}\neq\vn, B_{i+k'p^y\bmod p^{y+1}}\neq\vn.\]
  This means that the multiset $\frac{N}{p^{y+1}q}\cdot B_{i\bmod p^y}$ (which is supported in a $pq$-cycle) intersects nontrivially at least two $q$-cycles, namely 
  \[\frac{iN}{p^{y+1}q}+k\frac{N}{pq}+\frac{N}{q}\ZZ_N, \;\;\;\; \frac{iN}{p^{y+1}q}+k'\frac{N}{pq}+\frac{N}{q}\ZZ_N.\]
  By the second part of Lemma \ref{t1fail}, $\frac{N}{p^{y+1}q}\cdot B_{i\bmod p^y}$ must be supported on a $p$-cycle, in particular,
  \[\frac{iN}{p^{y+1}q}+\ell\frac{N}{q}+\frac{N}{p}\ZZ_N,\]
  for some $\ell$. But the latter implies that $B_{i\bmod p^y}$ is absorbed $\bmod p^yq$:
  \[B_{i\bmod p^y}=B_{i+\ell'p^y\bmod p^yq},\]
  where
  \[\ell'\equiv p\ell\bmod q.\]
  But this leads to a contradiction, since \eqref{edmodpxq} implies that $B_{i+kp^x\bmod p^xq}\neq \vn$ for every $k$, while
  \[B_{i+kp^x\bmod p^xq}\ssq B_{i+(kp^{x-y})p^y\bmod p^yq}=\vn\]
  when $kp^{x-y}\not\equiv\ell'\bmod q$, contradiction.
  
  Therefore, if $B_{i\bmod p^{x-1}}$ is absorbed $\bmod p^x$, then $B_{i\bmod p^y}$ is either absorbed of equidistributed $\bmod p^{y+1}$ for every $y<x$; this implies
  \[\abs{B_{i+kp^x\bmod p^xq}}=\frac{1}{q}\abs{B_{i\bmod p^x}}=\frac{\abs{B}}{p^tq},\]
  for some $t\in\NN$.
  With the same arguments as above, if $B_{i\bmod p^{x-1}}$ is equidistributed $\bmod p^x$ 
	we deduce that if $y\leq x-1$ is the smallest exponent such that $B_{i\bmod p^{y}}$ is neither absorbed nor equidistributed $\bmod p^{y+1}$
  (assuming there exists one), then $B_{i\bmod p^y}$ is absorbed $\bmod p^yq$, which means that every element in $B_{i\bmod p^y}$ has the same residue $\bmod q$,
	again by the second part of Lemma \ref{t1fail}. In this case, there is some $k$ such that
  \[\abs{B_{i+kp^y\bmod p^yq}}=\abs{B_{i\bmod p^y}}=\frac{\abs{B}}{p^s},\]
  for some $s\in\NN$. We write
  \[B_{i\bmod pq}=\bigsqcup_{k=1}^{p^{x-1}}B_{i+kpq\bmod p^xq},\]
  and we define $y(k)\leq x-1$ to be the smallest exponent such that $B_{i+kpq\bmod p^{y(k)}}$ is neither absorbed nor equidistributed $\bmod p^{y(k)+1}$; if such an exponent does not exist, define
  $y(k)=x-1$. We define the relation $k\sim k'$ to denote $y(k)=y(k')$ and $k\equiv k'\bmod p^{y(k)-1}$; it is not hard to see that this is an equivalence relation and 
  \[B_{i\bmod pq}=\bigsqcup_{k\in K}B_{i+kpq\bmod p^{y(k)-1}q}.\]
  As we have seen though, the cardinality of each set in the union above is divisible by $\frac{\abs{B}}{p^rq}$, for some $r\in\NN$. Reversing the roles of $p$ and $q$, we obtain that
  \begin{equation}\label{divisibility99}
  \frac{\abs{B}}{pq}\mid \abs{B_{i\bmod pq}},
  \end{equation}
  for every $i$. If $B$ were not equidistributed $\bmod pq$, there would exist some $i$ such that
  \[\abs{B_{i\bmod pq}}\geq \frac{2\abs{B}}{pq}.\]
  By \eqref{divisibility99} and the trivial partition
  \[B_{i\bmod q}=\bigsqcup_{j=1}^pB_{i+jq\bmod pq},\]
  we obtain the existence of a $j$ such that $B_{i+jq\bmod pq}=\vn$. Similarly, there is some $\ell$ such that $B_{i+\ell p\bmod pq}=\vn$. If $B_{i+kpq\bmod p^{x-1}}$ is absorbed $\bmod p^x$, then it is
  equidistributed $\bmod p^{x-1}q$, as we've shown above. But then, $B_{i+kpq+up^{x-1}\bmod p^{x-1}q}$ is nonempty for every $u$, however,
  \[B_{i+kpq+up^{x-1}\bmod p^{x-1}q}\ssq B_{i+\ell p\bmod pq}=\vn\]
  for $up^{x-1}\equiv \ell p\bmod pq$, contradiction. This shows that $B_{i+kpq\bmod p^{x-1}}$ is equidistributed $\bmod p^x$ for every $k$, showing that $\frac{\abs{B}}{p^t}$ divides $\abs{B_{i\bmod pq}}$.
  Similarly, we obtain that $\frac{\abs{B}}{q^s}$ divides $\abs{B_{i\bmod pq}}$, showing thus $\abs{B}\mid \abs{B_{i\bmod pq}}$, contradiction. Thus, $B$ is equidistributed $\bmod pq$, and the same holds for
  $A$, by reversing the roles of $A$ and $B$. This eventually shows
  \[A(\ze_{pq})=B(\ze_{pq})=0,\]
  as desired.
 \end{proof}

 \begin{rem}
  The above verifies property \ref{t2} for the pairs of prime powers $(p^m,q)$, $(p,q^n)$, and $(p,q)$.
 \end{rem}

 \bigskip\bigskip
 \section{Root deficit}
 \bigskip\bigskip

 We continue to assume that $N\in\Spq'$ and $(A,B)$ is a spectral pair with $A\in\Fmax(N)$, so that $B\in\Fmax(N)$ as well.
 
 \begin{defn}
  Let $A\ssq\ZZ_N$ satisfy $A(\ze_q)=A(\ze_p)=\dotsb=A(\ze_{p^k})=0\neq A(\ze_{p^{k+1}})$. We define the following property:
  \begin{enumerate}[{\bf(wT2)},leftmargin=*]
   \item\label{wt2} It holds $A(\ze_{pq})=\dotsb=A(\ze_{p^kq})=0$. 
  \end{enumerate}
 \end{defn}
 
 It should be emphasized that when \ref{wt2} holds, then $A$ is equidistributed $\bmod p^kq$ (and in particular, every $A_{j\bmod p^k}$ is equidistributed $\bmod p^kq$). Indeed, as the hypothesis of the Definition
 and \ref{wt2} imply that $\Phi_d(X)\mid A(X)$ for every $d\mid p^kq$, $d\neq1$. Therefore,
 \[X^{p^kq-1}+\dotsb+1\mid A(X),\]
 or equivalently, $A(X)=G(X)(X^{p^kq-1}+\dotsb+1)$ for some $G(X)\in\ZZ[X]$. When we reduce $\bmod(X^{p^kq}-1)$ we obtain
 \begin{align*}
  A(X) &\equiv G(X)(X^{p^kq-1}+\dotsb+1)\bmod(X^{p^kq}-1)\\
  &\equiv G(1)(X^{p^kq-1}+\dotsb+1)\bmod(X^{p^kq}-1),
 \end{align*}
 which clearly shows that
  \begin{equation}\label{wt2holds}
  \abs{A_{j\bmod p^k}}=\frac{\abs{A}}{p^k}, \;\;\;\; \abs{A_{j\bmod p^kq}}=\frac{\abs{A}}{p^kq}, \;\;\;\; \forall j.
  \end{equation}

 \begin{prop}\label{absfreefails}
  Let $N\in\Spq'$, $A\in\Fmax(N)$, satisfying 
  \[A(\ze_q)=A(\ze_p)=\dotsb=A(\ze_{p^k})=0\neq A(\ze_{p^{k+1}}).\]
  Then, there is some $j$ such that $A_{j\bmod p^k}$ is absorbed $\bmod p^kq$. In particular, \ref{wt2} fails for $A$.
 \end{prop}
 
 \begin{proof}
 By definition, $A$ is absorption-free, therefore, there is some $j$ such that $A_{j\bmod p^k}$ is not absorbed $\bmod p^{k+1}$. Hence,
\[A_{j\bmod p^k}=\bigcup_{i=1}^r A_{j+(\ell_iq) p^k\bmod p^{k+1}},\]
where $r\geq2$, the $\ell_i$ are pairwise distinct $\bmod p$ and all sets in the above union are nonempty.
This shows that the multiset $\frac{N}{p^{k+1}q}\cdot A_{j\bmod p^k}$ has elements on $r$ distinct $q$-cycles, namely
 \[\frac{jN}{p^{k+1}q}+\ell_i\frac{N}{p}+\frac{N}{q}\ZZ_N, \;\;\;\; 1\leq i\leq r.\]
Consider two arbitary elements from two different $q$-cycles (relabeling the $\ell_i$ if necessary), say
\[\frac{a_1N}{p^{k+1}q}=\frac{jN}{p^{k+1}q}+\ell_1\frac{N}{p}+m_1\frac{N}{q} \;\;\;\; \text{ and } 
\;\;\;\; \frac{a_2N}{p^{k+1}q}=\frac{jN}{p^{k+1}q}+\ell_2\frac{N}{p}+m_2\frac{N}{q},\]
where $a_1, a_2\in A_{j\bmod p^k}$. We will show that $m_1\equiv m_2\bmod q$; if not, then 
\[\frac{N}{p^{k+1}q}(a_2-a_1)=(\ell_2-\ell_1)\frac{N}{p}+(m_2-m_1)\frac{N}{q}\in\frac{N}{pq}\ZZ_N^\star,\]
or equivalently $a_2-a_1\in p^k\ZZ_N^\star$, or $\ord(a_2-a_1)=\frac{N}{p^k}$, which by Corollary \ref{specord} gives $B(\ze_N^{p^k})=0$, contradicting the hypothesis 
$A(\ze_{p^{k+1}})\neq0$, due to Lemma \ref{rootpatterns} (especially \eqref{ABpsymalt}, with the roles of $A$ and $B$ reversed).
Thus, $m_1\equiv m_2\bmod q$, and considering any other pair of elements of $\frac{N}{p^{k+1}q}\cdot A_{j\bmod p^k}$ on two different $q$-cycles, where one element is
either $\frac{a_1N}{p^{k+1}q}$ or $\frac{a_2N}{p^{k+1}q}$, we obtain
\[\frac{N}{p^{k+1}q} A_{j\bmod p^k}\ssq \frac{jN}{p^{k+1}q}+\frac{N}{p}\ZZ_N+m_1\frac{N}{q},\]
by repeating the above argument.
This shows that $\frac{N}{p^{k+1}q}\cdot A_{j\bmod p^k}$ is supported on a $p$-cycle, thus $A_{j\bmod p^k}$ is absorbed $\bmod p^kq$, as desired.
The fact that \ref{wt2} fails follows easily from \eqref{wt2holds}.
 \end{proof}

 \begin{cor}\label{rootssmallpowers}
  Let $N\in\Spq'$, $(A,B)$ a spectral pair in $\ZZ_N$ with $A\in\Fmax(N)$. It holds
  \[A(\ze_{q^2})=B(\ze_{q^2})=A(\ze_{p^2})=B(\ze_{p^2})=0.\]
  If $p<q$, then 
  \[A(\ze_{p^3})=\dotsb=A(\ze_{p^{\ceil{\log_pq}+1}})=B(\ze_{p^3})=\dotsb=B(\ze_{p^{\ceil{\log_pq}+1}})=0,\]
  as well. Moreover,
  \[A(\ze_N^q)=B(\ze_N^q)=A(\ze_N^p)=\dotsb=A(\ze_N^{p^{\ceil{\log_pq}}})=B(\ze_N^p)=\dotsb=B(\ze_N^{p^{\ceil{\log_pq}}})=0.\]
 \end{cor}
 
 \begin{proof}
  By Proposition \ref{absfreefails}, property \ref{wt2} fails for both $A$ and $B$. Combined with Proposition \ref{basict2}, this shows that 
  \[A(\ze_{p^2})=B(\ze_{p^2})=A(\ze_{q^2})=B(\ze_{q^2})=0,\]
  and then Lemma \ref{rootpatterns} implies
  \[A(\ze_N^p)=B(\ze_N^p)=A(\ze_N^q)=B(\ze_N^q)=0.\]
  Next, assume $p<q$, and suppose
  \[A(\ze_p)=\dotsb=A(\ze_{p^k})=0\neq A(\ze_{p^{k+1}}).\]
  By Proposition \ref{absfreefails}, there is some $j$ such that $A_{j\bmod p^k}$ is absorbed $\bmod p^kq$, that is,
  \[A_{j\bmod p^k}=A_{j+\ell p^k\bmod p^kq}\]
  for some $\ell$. But then,
  \[A_{j\bmod p^k}\ssq A_{j+\ell p^k\bmod pq},\]
  therefore,
  \[\frac{\abs{A}}{p^k}=\abs{A_{j\bmod p^k}}\leq \abs{A_{j+\ell p^k\bmod pq}}=\frac{\abs{A}}{pq},\]
  which yields $p^{k-1}> q$ or equivalently, $k>\log_pq+1$, so finally
  \[k\geq\ceil{\log_pq}+1.\]
  Reversing the roles of $A$ and $B$, we obtain the same result for $B$. The final part then follows directly from Lemma \ref{rootpatterns}.
 \end{proof}

 \begin{lemma}\label{Ulb}
  Let $(A,B)$ be a spectral pair in $\ZZ_N$, where $N\in\Spq'$ and $A\in\Fmax(N)$. Then, neither $A$ nor $B$ satisfy \ref{wt1}, as 
  \[\abs{U_A^N(p)}, \abs{U_B^N(p)}\geq \ceil{\log_pq}\]
  and
  \[\abs{U_A^N(q)}, \abs{U_B^N(q)}\geq \ceil{\log_qp}.\]
 \end{lemma}
 
 \begin{proof}
 Without loss of generality, we may assume $p<q$, so that the second inequality becomes $\abs{U_A^N(q)}, \abs{U_B^N(q)}\geq1$, which easily follows from Proposition \ref{maxpqroots}.
  We remind that $m_p(A)$ denotes as usual the maximal exponent $x$ such that $A(\ze_N^{p^x})\neq0$. We will show
  \[\abs{U_B^N(p)}\geq\ceil{\log_pq},\]
  and the inequality $\abs{U_A^N(p)}\geq\ceil{\log_pq}$ has similar proof. Denote
  \[s=\abs{[0,m_p(A)]\cap R_A^N(p)}.\]
  We will show the existence of $j$ satisfying
  \[\abs{A_{j\bmod p^{m-m_p(A)-1}q^{n-1}}}=p^sq\]
  and $A_{j\bmod p^{m-m_p(A)-1}q^{n-1}}$ is not absorbed $\bmod p^{m-m_p(A)}q^{n-1}$. For convenience, we will put $d=p^{m-m_p(A)-1}q^{n-1}$. Suppose that
  \[A(\ze_N)=A(\ze_N^{p})=\dotsb=A(\ze_N^{p^k})=0\neq A(\ze_N^{p^{k+1}}),\]
  so that
  \[B(\ze_p)=B(\ze_{p^2})=\dotsb=B(\ze_{p^{k+1}})=0\neq B(\ze_{p^{k+2}}),\]
  by Lemma \ref{rootpatterns}. It holds
  \[A(X)\equiv P(X)\Php+Q(X)\Phq\bmod(X^N-1)\]
  for some $P(X),Q(X)\in\ZZ_{\geq0}[X]$ by Theorem \ref{vansumspos}, and $P\not\equiv0$ by Theorem \ref{summary}\eqref{summary2}.
  Replacing $X$ by $X^p$, and observing that $\Phq\equiv\Phi_q(X^{pN/q})\bmod(X^N-1)$ we obtain
  \[A(X^p)\equiv pP(X^p)+Q(X^p)\Phq\bmod(X^N-1).\]
  Since $A(\ze_N^p)=0$, we must have
  \[pP(\ze_N^p)=A(\ze_N^p)-Q(\ze_N^p)\Phi_q(\ze_q^p)=0,\]
  therefore
  \[P(X^p)\equiv \wt{P}(X^p)\Php+\wt{Q}(X^p)\Phq\bmod(X^N-1),\]
  where $\wt{P}(X),\wt{Q}(X)\in\ZZ_{\geq0}[X]$, hence
  \[A(X^p)\equiv pP_1(X^p)\Php+Q_1(X^p)\Phq\bmod(X^N-1),\]
  where $P_1(X^p)=\wt{P}(X^p)$ and $Q_1(X^p)=Q(X^p)+p\wt{Q}(X^p)$. The polynomial $P_1$ is not identically zero, since $A(\ze_N^{p^{k+1}})\neq0$ (see Theorem \ref{vansumspos}). Repeating
  this process for the roots 
  \[\ze_N^{p^2},\dotsc,\ze_N^{p^k},\]
  we obtain
  \begin{equation}\label{pk104}
  A(X^{p^k})\equiv p^kP_k(X^{p^k})\Php+Q_k(X^{p^k})\Phq\bmod(X^N-1),
  \end{equation}
  where $P_k(X),Q_k(X)\in\ZZ_{\geq0}[X]$, and $P_k$ is not identically zero, since $A(\ze_N^{p^{k+1}})\neq0$. As 
  \begin{equation}\label{pk+1104}
  A(X^{p^{k+1}})\equiv p^{k+1}P_k(X^{p^{k+1}})+Q_k(X^{p^{k+1}})\Phq\bmod(X^N-1)
  \end{equation}
  is the mask polynomial of the mutliset $p^{k+1}\cdot A$, we deduce that
  there are elements in $p^{k+1}\cdot A$ whose multiplicity is exactly $p^{k+1}$ (in general, since $A$ is a proper set, the multiplicity of any element in $d\cdot A$ cannot exceed $d$). The element
  $p^{k+1}a$ can only have multiplicity $p^{k+1}$ in $A$ only if
  \[a+\frac{N}{p^{k+1}}\ZZ_N\ssq A,\]
  i.e. $a$ is an element of a $p^{k+1}$-cycle contained in $A$, or equivalently, $\abs{A_{a\bmod p^{m-k-1}q^n}}=p^{k+1}$. $A$ cannot have two $p^{k+1}$-cycles of the form
  \[a+\frac{N}{p^{k+1}}\ZZ_N, \;\;\;\; a+\ell\frac{N}{p^{k+1}q}+\frac{N}{p^{k+1}}\ZZ_N,\]
  for $\ell\not\equiv0\bmod q$, because in that case, we would have 
  \[\ell p^{m-k-1}q^{n-1}+p^{m-k-1}q^n\ZZ_N\ssq A-A,\]
  which implies that
  \[(A-A)\cap d\znp\neq0, \;\;\;\; \forall d\in\set{p^{m-k-1}q^{n-1},p^{m-k}q^{n-1},\dotsc,p^{m-2}q^{n-1},p^{m-1}q^{n-1}},\]
  whence
  \[B(\ze_{pq})=B(\ze_{p^2q})=\dotsb=B(\ze_{p^{k+1}q})=0,\]
  by Theorem \ref{mainref}\ref{i},
  which means that $B$ satisfies \ref{wt2}, contradicting Proposition \ref{absfreefails}. The above show that
  \begin{equation}\label{104absorption}
  \abs{A_{a\bmod p^{m-k-1}q^n}}=p^{k+1}\Longrightarrow A_{a\bmod p^{m-k-1}q^{n-1}}\text{ is absorbed }\bmod p^{m-k-1}q^n.
  \end{equation}
  Now let 
  \[R_A^N(p)\cap[0,m_p(A)]=\set{r_0,r_1,\dotsc,r_{s-1}},\]
  where the $r_i$ are in increasing order. We know already that $r_i=i$, for $0\leq i\leq k$. Continuing from \eqref{pk104} we obtain
  \[A(X^{p^{r_{k+1}}})\equiv p^{k+1}P_k(X^{p^{r_{k+1}}})+Q_k(X^{p^{r_{k+1}}})\Phq\bmod(X^N-1).\]
  Since $A(\ze_N^{p^{r_{k+1}}})=0\neq A(\ze_N^{p^{m_p(A)}})$, we get that $P_k(\ze_N^{p^{r_{k+1}}})=0\neq P_k(\ze_N^{p^{m_p(A)}})$, whence by Theorem \ref{mainref}\ref{i} we obtain
  \[P_k(X^{p^{r_{k+1}}})\equiv\wt{P}_k(X^{p^{r_{k+1}}})\Php+\wt{Q}_k(X^{p^{r_{k+1}}})\Phq\bmod(X^N-1),\]
  for some $\wt{P}_k(X), \wt{Q}_k(X)\in\ZZ_{\geq0}[X]$, with $\wt{P}_k$ not identically zero. This gives
  \[A(X^{p^{r_{k+1}}})\equiv p^{k+1}P_{k+1}(X^{p^{r_{k+1}}})\Php+Q_{k+1}(X^{p^{r_{k+1}}})\Phq\bmod(X^N-1),\]
  where $P_{k+1}(X)=\wt{P}_k(X)$ and $Q_{k+1}(X)=Q_k(X)+p^{k+1}\wt{Q}_k(X)$. Repeating this process for the roots
  \[\ze_N^{p^{r_{k+2}}},\dotsc,\ze_N^{p^{r_{s-1}}},\]
  we obtain
  \[A(X^{p^{r_{s-1}}})\equiv p^{s-1}P_{s-1}(X^{p^{r_{s-1}}})\Php+Q_{s-1}(X^{p^{r_{s-1}}})\Phq\bmod(X^N-1),\]
  for some $P_{s-1}(X),Q_{s-1}(X)\in\ZZ_{\geq0}[X]$, where $P_{s-1}$ is not identically zero, since $r_{s-1}<m_p(A)$ and $A(\ze_N^{p^{m_p(A)}})\neq0$; this is achieved by repeated use of Theorem
  \ref{vansumspos}. Therefore, it holds
  \[A(X^{p^{m_p(A)+1}})\equiv p^sP_{s-1}(X^{p^{m_p(A)+1}})+Q_{s-1}(X^{p^{m_p(A)+1}})\Phq\bmod(X^N-1).\]
  By definition of $m_p(A)$ we have
  \[A(\ze_N^{p^{m_p(A)+1}})=A(\ze_N^{p^{m_p(A)+2}})=\dotsb=A(\ze_N^{p^m})=0,\]
  hence
  \[P_{s-1}(\ze_N^{p^{m_p(A)+1}})=P_{s-1}(\ze_N^{p^{m_p(A)+2}})=\dotsb=P_{s-1}(\ze_N^{p^{m}})=0,\]
  so by Lemma \ref{ma} we get
  \[P_{s-1}(X^{p^{m_p(A)+1}})\equiv P_s(X^{p^{m_p(A)+1}})\Phq\bmod(X^N-1),\]
  where $P_s(X)\in\ZZ_{\geq0}[X]$, hence
  \begin{equation}\label{104final}
  A(X^{p^{m_p(A)+1}})\equiv(p^sP_s(X^{p^{m_p(A)+1}})+Q_{s-1}(X^{p^{m_p(A)+1}}))\Phq\bmod(X^N-1),   
  \end{equation}
  so that the multiset $p^{m_p(A)+1}\cdot A$ is a disjoint union of $q$-cycles. It is clear from \eqref{104final} that there is some $j$ with $\abs{A_{j\bmod p^{m-m_p(A)-1}q^n}}\geq p^s$. However, we will show that
  \[\abs{A_{j\bmod p^{m-m_p(A)-1}q^n}}\leq p^s,\]
  for every $j$. For this purpose, consider the $p$-adic expansion of $a-j$ where $a$ is an arbitrary element of $A_{j\bmod p^{m-m_p(A)-1}q^n}$
  \[a-j\equiv a_{m-m_p(A)-1}p^{m-m_p(A)-1}+\dotsb+a_{m-1}p^{m-1}\bmod p^m,\]
  where as usual, $0\leq a_i\leq p-1$, $m-m_p(A)-1\leq i\leq m-1$. Now, if $\abs{A_{j\bmod p^{m-m_p(A)-1}q^n}}> p^s$, then there would exist two elements 
  $a,a'\in A_{j\bmod p^{m-m_p(A)-1}q^n}$, such that the $p$-adic expansions
  of $a-j$ and $a'-j$ would have the same digits at all places $m-r$, for $r-1\in R_A^N(p)$. Therefore, $a-a'=(a-j)-(a'-j)\in d\znp$, where $d=p^{m-r}q^n$, with $r-1\leq m_p(A)$ and $r-1\notin R_A^N(p)$. By Lemma
  \ref{rootpatterns} we would have $r\notin S_B^N(p)$ whence $B(\ze_{p^r})\neq0$, contradicting Theorem \ref{mainref}\ref{i} or Corollary \ref{specord}, which yield $B(\ze_{p^r})=0$ due to
  $a-a'\in p^{m-r}q^n\znp$.
  
  To summarize, we have $\abs{A_{j\bmod p^{m-m_p(A)-1}q^n}}\leq p^s$ for every $j$, and equality occurs for some $j$. Using the same arguments as above, we deduce that if $\abs{A_{j\bmod p^{m-m_p(A)-1}q^n}}=p^s$, then
  the elements of $A_{j\bmod p^{m-m_p(A)-1}q^n}-j$ obtain all possible arrangements of $p$-adic digits at the places $m-r$, where $r-1\in R_A^N(p)$; in particular, $A_{j\bmod p^{m-m_p(A)-1}q^n}$ must be a disjoint
  union of $p^{k+1}$-cycles. So, let $\abs{A_{j\bmod p^{m-m_p(A)-1}q^n}}= p^s$. By \eqref{104absorption} and \eqref{104final} we obtain
  \[\abs{A_{j\bmod p^{m-k-1}q^{n-1}}}=p^{k+1}, \;\;\;\; \abs{A_{j\bmod p^{m-m_p(A)-1}q^{n-1}}}=p^sq.\]
  We observe that $A_{j\bmod p^{m-m_p(A)-1}q^{n-1}}$ is equidistributed $\bmod p^{m-m_p(A)-1}q^n$, while $A_{j\bmod p^{m-k-1}q^{n-1}}$ is absorbed $\bmod p^{m-k-1}q^n$. For every $r\in[k,m_p(A)]$, we will take into account
  the number of nonempty sets in the partition
  \[A_{j\bmod p^{m-r-1}q^{n-1}}=\bigsqcup_{u=0}^{q-1}A_{j+u p^{m-r-1}q^{n-1}\bmod p^{m-r-1}q^n},\]
  denoted by $f_j(r)$, so that $1\leq f_j(r)\leq q$ for all $r\in[k,m_p(A)]$, and $f_j(m_p(A))=q$, $f_j(k)=1$. Now,
  pick $j'$ such that $j'\equiv j\bmod p^{m-m_p(A)-1}q^n$ and
  \[f_{j'}(r-1)=\max_{0\leq\ell\leq p-1}f_{j'+\ell p^{m-r-1}q^{n-1}}(r-1).\]
  Without loss of generality, we may simply have $j'=j$, and it holds
  \[f_j(r-1)\geq\frac{1}{p} f_j(r), \;\;\;\; r\in[k+1,m_p(A)].\]
  We observe that if $f_j(r-1)<f_j(r)$, then $r+1\in U_B^N(p)$. Indeed, consider the multiset $p^r\cdot A_{j\bmod p^{m-r-1}q^{n-1}}$, which is supported at the $pq$-cycle $p^rj+\frac{N}{pq}\ZZ_N$. We will
  show that $A(\ze_N^{p^r})\neq0$, which implies $B(\ze_{p^{r+1}})\neq0$ by Lemma \ref{rootpatterns}. Assume otherwise, that $A(\ze_N^{p^r})=0$; hence, $p^r\cdot A_{j\bmod p^{m-r-1}q^{n-1}}$ must be 
  a disjoint union of $p$- and $q$-cycles. Since $f_j(r-1)<f_j(r)$, we cannot have any $q$-cycles, because
  \[f_{j+\ell p^{m-r-1}q^{n-1}}(r-1)\leq f_j(r-1)<f_j(r)\leq f_j(m_p(A))=q, \;\;\;\; 0\leq\ell\leq p-1,\]
  which means that for every $0\leq \ell\leq p-1$, there is $0\leq u\leq q-1$, such that
  \[A_{j+\ell p^{m-r-1}q^{n-1}+u p^{m-r}q^{n-1}\bmod p^{m-r}q^n}=\vn.\]
  But $p^r\cdot A_{j\bmod p^{m-r-1}q^{n-1}}$ cannot be a disjoint union of $p$-cycles either, because in this case $A_{j\bmod p^{m-r-1}q^{n-1}}$ would be equidistributed $\bmod p^{m-r}q^{n-1}$, whence
  $f_j(r-1)=f_j(r)$, contradiction. Hence, $A(\ze_N^{p^r})B(\ze_{p^{r+1}})\neq0$. Moreover, $p^r\cdot A_{j\bmod p^{m-r-1}q^{n-1}}$ can be supported neither on a $q$-cycle, because in this case $f_j(r-1)=f_j(r)$ and
  \[f_{j+\ell p^{m-r-1}q^{n-1}}(r-1)=0, \;\;\;\; 1\leq \ell \leq p-1,\]
  nor on a $p$-cycle, because it would hold
  \[f_{j+\ell p^{m-r-1}q^{n-1}}(r-1)=f_j(r)=1, \;\;\;\; 0\leq \ell\leq p-1,\]
  contradiction. Hence, by Proposition \ref{preroothunt}, we obtain
  \[(p^rA-p^rA)\cap\frac{N}{pq}\znp\neq\vn,\]
  so there are $a,a'\in A$ such that $a-a'\in\frac{N}{p^{r+1}q}$, whence $B(\ze_{p^{r+1}q})=0$ by Theorem \ref{mainref}\ref{i}, and $r+1\in U_B^N(p)$, as desired.
  
  For the last part of our proof, we will just show that we have $f_j(r-1)<f_j(r)$ for at least $\ceil{\log_pq}$ values of $r$. Indeed, let 
  \[\set{r_1,\dotsc,r_t}=\set{r\in[k,m_p(A)]:f_j(r-1)<f_j(r)}.\]
  By definition, and the property $f_j(r-1)\geq\frac{1}{p}f_j(r)$ for all $r$, we obtain
  \begin{multline*}
  q=f_j(r_t)\leq pf_j(r_t-1)=pf_j(r_{t-1})\leq p^2 f_j(r_{t-1}-1)=p^2 f_j(r_{t-2})\leq \dotsb\\
  \dotsb\leq p^{t-1}f_j(r_1)\leq p^tf_j(r_1-1)=p^tf_j(k)=p^t,
  \end{multline*}
  which clearly implies $t\geq\ceil{\log_pq}$, finally showing
  \[\abs{U_B^N(p)\cap[1,m_p(A)+1]}\geq \ceil{\log_pq},\]
  as desired.
 \end{proof}

 \begin{cor}\label{divAB}
  Let $(A,B)$ be a spectral pair in $\ZZ_N$, where $N\in\Spq'$ and $A\in\Fmax(N)$. Then
  \[p^{\abs{S_B^N(p)}+\ceil{\log_pq}}q^{\abs{S_B^N(q)}+\ceil{\log_qp}}\mid \abs{B}\]
  and
  \[p^{\abs{S_A^N(p)}+\ceil{\log_pq}}q^{\abs{S_A^N(q)}+\ceil{\log_qp}}\mid \abs{A}.\]
 \end{cor}
 
 \begin{proof}
  By Lemmata \ref{t1fail} and \ref{Ulb} we obtain
  \[p^{\abs{S_B^N(p)}+\ceil{\log_pq}}\mid \abs{B}.\]
  Reversing the roles of $A$ and $B$, and of $p$ and $q$, we obtain the desired result.
 \end{proof}
 
 If $p<q$, we observe that the power of $p$ dividing $\abs{A}$ with $A\in\Fmax(N)$ is at least
 \[\abs{S_A^N(p)}+\ceil{\log_pq}\geq 2\ceil{\log_pq}+2\geq6,\]
 which can be seen by Theorem \ref{summary} and Corollary \ref{rootssmallpowers}. Moreover, the power of $q$ dividing $\abs{A}$ is at least
 \[\abs{S_A^N(q)}+1\geq 4,\]
 again by Theorem \ref{summary} and Corollary \ref{rootssmallpowers}. We have thus proven:
 
 \begin{cor}
  Let $N=p^mq^n$, with $p<q$, and let $A\ssq\ZZ_N$ be spectral. If either $m\leq6$ or $n\leq4$, then $A$ tiles. The same conclusion holds if $p^{m-4}<q^2$.
 \end{cor}
 
 \begin{proof}
  Assume otherwise, that $N\in\Spq$; without loss of generality, we may assume $N\in\Spq'$.
  Let $A\in\Fmax(N)$, so that $p^6q^4\mid\abs{A}$, as shown above. By Proposition \ref{maxpower}, this would mean that
  $A$ tiles, if either $m\leq6$ or $n\leq4$, contradiction. For the second part, we observe that $p^{m-4}<q^2$ is equivalent to $p^{\frac{m}{2}-2}<q$, or equivalently,
  \[\frac{m}{2}-1\leq\ceil{\log_pq}\Lrar m\leq 2\ceil{\log_pq}+2\leq \abs{S_A^N(p)}+\ceil{\log_pq},\]
  which again implies that $p^m\mid\abs{A}$, hence $A$ tiles by Proposition \ref{maxpower}, contradiction. This completes the proof.
 \end{proof}

 We may assume henceforth that both exponents are at least $5$, and that of the smallest prime is at least $7$. We will refine the above bound, by introducing a new concept.

 \begin{defn}
  Let $(A,B)$ be a primitive spectral pair in $\ZZ_N$. Define the {\em $p$-root deficit} of $A$ as
  \[\deficit_p(A)=\abs*{\set{x:A(\ze_{p^{x+1}})\neq0=B(\ze_N^{p^x})}},\]
  and similarly define $\deficit_q(A)$, $\deficit_p(B)$, $\deficit_q(B)$.
 \end{defn}
 
 \begin{rem}
  We should note that the root deficits of a spectral set $A$ always depend on the choice of the spectrum $B$.
 \end{rem}
 
 \begin{lemma}\label{deficitestimate}
  Let $(A,B)$ be a primitive spectral pair in $\ZZ_N$, $N\in\Spq'$, $A\in\Fmax(N)$. Then, the following inequalities hold:
  \begin{eqnarray*}
\abs{A}=\abs{B}\leq\min\{p^{\abs{S_B^N(p)}+\deficit_p(B)}q^{\abs{S_A^N(q)}},p^{\abs{S_A^N(p)}+\deficit_p(A)}q^{\abs{S_B^N(q)}},\\ 
  p^{\abs{S_B^N(p)}}q^{\abs{S_A^N(q)}+\deficit_q(A)},  p^{\abs{S_A^N(p)}}q^{\abs{S_B^N(q)}+\deficit_q(B)}\}
  \end{eqnarray*}
 \end{lemma}
 
 \begin{proof}
  We will show that
  \[\abs{A}\leq p^{\abs{S_B^N(p)}}q^{\abs{S_A^N(q)}+\deficit_q(A)}.\]
  All the other inequalities are obtained by reversing the roles of $p$, $q$, and of $A$, $B$. First, we will prove that
  \begin{equation}\label{eq108}
  \abs{A_{j\bmod q^n}}\leq p^{\abs{S_B^N(p)}},
  \end{equation}
  for all $j$. Let 
  \[S_B^N(p)=\set{x_1,\dotsc,x_t}\ssq\set{1,2,\dotsc,m},\]
  where $t=\abs{S_B^N(p)}$ and the $x_i$ are in increasing order. We will write every element of $a\in A_{j\bmod q^n}$ in $p$-adic representation, that is
  \[a\equiv a_0+a_1p+\dotsb+a_{m-1}p^{m-1}\bmod p^m,\]
  where as usual $0\leq a_i\leq p-1$, $0\leq i\leq m-1$. No two elements of $A_{j\bmod q^n}$ have the same $p$-adic digits, otherwise $N=p^mq^n$ would divide their difference, i.e. they would be equal.
  Moreover, two distinct elements $a,a'\in A_{j\bmod q^n}$ cannot have the same $p$-adic digits at the places
  \[m-x_t,\dotsc,m-x_1,\]
  otherwise, $a-a'\in p^{m-x}q^n\znp$ where $x\notin S_B^N(p)$, and as a consequence
  \[B(\ze_{\ord(a-a')})=B(\ze_{p^x})\neq0,\]
  contradicting Corollary \ref{specord}. Thus, the elements of $A_{j\bmod q^n}$ have distinct digits at one of the places $m-x_i$. This clearly shows that
  \[\abs{A_{j\bmod q^n}}\leq p^t=p^{\abs{S_B^N(p)}},\]
  proving our claim. Next, we will show that there is some $j$ such that
  \[\abs{A_{j\bmod q^n}}\geq \frac{\abs{A}}{q^{\abs{S_A^N(q)}+\deficit_q(A)}}.\]
  Recall that $m_q(B)$ is the maximal integer $x\in[0,n-1]$ for which
  \[B(\ze_N^{q^x})\neq0.\]
  Since $B\in\Fmax(N)$, such an exponent exists, as we have already seen. Furthermore, Lemma \ref{rootpatterns} and Proposition \ref{maxpqroots} imply that
  \[A(\ze_{q^{m_q(B)+1}})\neq0=A(\ze_{pq^{m_q(B)+1}}).\]
  Applying Lemma \ref{abseqd} for $d=q^{m_q(B)}$, with the roles of $p,q$ and $A,B$ reversed, we obtain that $A_{j\bmod q^{m_q(B)}}$ is either absorbed or equidistributed $\bmod q^{m_q(B)+1}$, for every $j$.
  Since $A(\ze_{q^{m_q(B)+1}})\neq0$, there must exist some $j$ such that $A_{j\bmod q^{m_q(B)}}$ is absorbed $\bmod q^{m_q(B)+1}$. Without loss of generality, we put
  \[A_{j\bmod q^{m_q(B)}}=A_{j\bmod q^{m_q(B)+1}}.\]
  The set $A_{j\bmod q^{m_q(B)+1}}$ is in turn equidistributed $\bmod pq^{m_q(B)+1}$, since the multiset $\frac{N}{pq^{m_q(B)+1}}\cdot A_{j\bmod q^{m_q(B)}}$ is contained in a $pq$-cycle for every $j$, and for this
  specific value of $j$ it is supported on a $p$-cycle. In particular,
  \[\abs{A_{j+\ell q^{m_q(B)+1}\bmod pq^{m_q(B)+1}}}=\frac{1}{p}\abs{A_{j\bmod q^{m_q(B)+1}}},\]
  for every $\ell$. For any $y\leq m_q(B)$ with $y+1\in S_A^N(q)$ it holds
  \begin{equation}\label{ed108}
  \abs{A_{j\bmod q^{y+1}}}=\frac{1}{q}\abs{A_{j\bmod q^y}},
  \end{equation}
  while if $y\leq m_q(B)$ and $y+1\notin S_A^N(q)$ it holds
  \begin{equation}\label{abs108}
  \abs{A_{j\bmod q^{y+1}}}=\abs{A_{j\bmod q^y}}.
  \end{equation}
  Indeed, since $B(\ze_N^{q^y})\neq0$ by Lemma \ref{rootpatterns}, we must have $a-a'\notin q^y\znp$ for every $a,a'\in A$, $a\neq a'$, due to Theorem \ref{mainref}\ref{i} (or Corollary \ref{specord}).
  If $A_{j\bmod q^y}$ were not absorbed $\bmod q^{y+1}$, there would exist some $u\not\equiv 0\bmod q$, such that $A_{j+u q^y\bmod q^{y+1}}\neq\vn$, hence for some $0\leq \ell\leq p-1$ satisfying
  \[\ell q^{m_q(B)+1-y}\not\equiv u\bmod p,\]
  and $a\in A_{j+\ell q^{m_q(B)+1}\bmod pq^{m_q(B)+1}}$, $a'\in A_{j+uq^y\bmod q^{y+1}}$, we obtain
  \[a-a'\equiv q^y(\ell q^{m_q(B)+1-y}-u)\bmod pq^{m_q(B)+1},\]
  hence $a-a'\in q^y\znp$, contradiction. Thus, combining \eqref{ed108} and \eqref{abs108} we get
  \[\abs{A_{j\bmod q^{m_q(B)+1}}}=\frac{\abs{A}}{q^{\abs{S_A^N(q)\cap[1,m_q(B)+1]}}}.\]
  Next, we recall
  \[\deficit_q(A)=\abs*{\set{y\in[1,n]:A(\ze_{p^y})\neq0=B(\ze_N^{p^{y-1}})}},\]
  where the set of $y$ in the description above must satisfy $y>m_q(B)+1$; in other words, every $y\in[m_q(B)+2,n]$ either satisfies $A(\ze_{q^y})=0$ or $A(\ze_{p^y})\neq0=B(\ze_N^{p^{y-1}})$, hence
  \[n-m_q(B)-1=\abs{S_A^N(q)\cap[m_q(B)+2,n]}+\deficit_q(A).\]
  Therefore, by \eqref{eq108} we get
  \begin{align*}
	p^{\abs{S_B^N(p)}} &\geq\max_{\ell}\abs{A_{j+\ell q^{m_q(B)+1}\bmod q^n}}\\
	&\geq\frac{1}{q^{n-m_q(B)-1}}\abs{A_{j\bmod q^n}}\\
	&=  \frac{\abs{A}}{q^{\abs{S_A^N(q)\cap[m_q(B)+2,n]}+\deficit_q(A)}q^{\abs{S_A^N(q)\cap[1,m_q(B)+1]}}},
	\end{align*}
  whence we finally obtain
  \[\abs{A}\leq p^{\abs{S_B^N(p)}}q^{\abs{S_A^N(q)}+\deficit_q(A)},\]
  as desired.
 \end{proof}

 \begin{cor}\label{deficitestimate2}
  Let $(A,B)$ be a spectral pair in $\ZZ_N$, where $N\in\Spq'$, $A\in\Fmax(N)$. Then,
  \[p^{\ceil{\log_pq}}q^{\ceil{\log_qp}}<\min\set{p^{\deficit_p(A)},p^{\deficit_p(B)},q^{\deficit_q(A)},q^{\deficit_q(B)}}.\]
  If in addition $p<q$, we get
  \[\deficit_q(A),\deficit_q(B)\geq3,\]
  and
  \[\deficit_p(A), \deficit_p(B)\geq 2\ceil{\log_pq}\geq4.\]
 \end{cor}
 
 \begin{proof}
  Denote
  \[s_p=\max\set{\abs{S_A^N(p)},\abs{S_B^N(p)}}, \;\;\;\; s_q=\max\set{\abs{S_A^N(q)},\abs{S_B^N(q)}}.\]
  By Corollary \ref{divAB} we obtain
  \[p^{s_p+\ceil{\log_pq}}q^{s_q+\ceil{\log_qp}}\mid\abs{A},\]
  whence by Lemma \ref{deficitestimate} it follows
  \begin{align*}
  p^{s_p+\ceil{\log_pq}}q^{s_q+\ceil{\log_qp}}\leq \min\{p^{\abs{S_B^N(p)}+\deficit_p(B)}q^{\abs{S_A^N(q)}},p^{\abs{S_A^N(p)}+\deficit_p(A)}q^{\abs{S_B^N(q)}},\\ 
  p^{\abs{S_B^N(p)}}q^{\abs{S_A^N(q)}+\deficit_q(A)},  p^{\abs{S_A^N(p)}}q^{\abs{S_B^N(q)}+\deficit_q(B)}\},
  \end{align*}
  which easily yields the inequality
  \[p^{\ceil{\log_pq}}q^{\ceil{\log_qp}}<\min\set{p^{\deficit_p(A)},p^{\deficit_p(B)},q^{\deficit_q(A)},q^{\deficit_q(B)}}.\]
  Next, suppose $p<q$, so that the above inequality becomes
  \[p^{\ceil{\log_pq}}q<\min\set{p^{\deficit_p(A)},p^{\deficit_p(B)},q^{\deficit_q(A)},q^{\deficit_q(B)}}.\]
  Since $q^2<p^{\ceil{\log_pq}}q$, the above yields
  \[\deficit_q(A),\deficit_q(B)\geq3\]
  and
  \[p^{2\ceil{\log_pq}}\leq\min\set{p^{\deficit_p(A)},p^{\deficit_p(B)}},\]
  or equivalently,
  \[\deficit_p(A), \deficit_p(B)\geq 2\ceil{\log_pq}\geq4,\]
  as desired.
 \end{proof}

 \begin{prop}\label{deficitestimate3}
  Let $(A,B)$ be a spectral pair in $\ZZ_N$, such that $N\in\Spq'$, $A\in\Fmax(N)$ and $p<q$. Then
  \[\deficit_p(A),\deficit_p(B)\leq m-2-2\ceil{\log_pq}\]
  and
  \[\deficit_q(A),\deficit_q(B)\leq n-4.\]
 \end{prop}
 
 \begin{proof}
  We remind that $\abs{S_A^N(p)-1}$, $\abs{U_A^N(p)-1}$ and $\deficit_p(A)$, are cardinalities of pairwise disjoint subsets of $\set{0,1,\dotsc,m-1}$, by definition. The bounds given by 
  Proposition \ref{basict2}, Corollary \ref{rootssmallpowers}, Lemma \ref{Ulb}, and Corollary \ref{deficitestimate2}, yield
  \[m\geq\abs{S_A^N(p)}+\abs{U_A^N(p)}+\deficit_p(A)\geq 2\ceil{\log_pq}+2+\deficit_p(A),\]
  which easily gives the first estimate (the same holds for $\deficit_p(B)$). For the second inequality we obtain similarly
  \[n\geq\abs{S_A^N(q)}+\abs{U_A^N(q)}+\deficit_q(A)\geq 4+\deficit_q(A),\]
  which gives the second inequality (the same inequality also holds for $\deficit_q(B)$).
 \end{proof}

 \bigskip\bigskip
 
 \section{Proof of Theorem \ref{mainthm}}
 \bigskip\bigskip
 
 From Corollary \ref{deficitestimate2} and Proposition \ref{deficitestimate3}, we obtain that $\F(N)$ and hence $\Fmax(N)$ is nonempty with $N\in\Spq'$, only if
 \[4\leq2\ceil{\log_pq}\leq m-2-2\ceil{\log_pq}\leq m-6,\]
 and $3\leq n-4$. This proves that if $N=p^mq^n$ with $p<q$,
 and either $m\leq9$ or $n\leq6$ holds, then every spectral subset of $\ZZ_N$ tiles, i.e. Fuglede's conjecture holds for cyclic groups of order $N$ with the given conditions. For the second part of the
 Theorem, 
 we invoke the inequalities involving $\deficit_p(A)$ and $\deficit_p(B)$ in Corollary \ref{deficitestimate2} and Proposition \ref{deficitestimate3} to obtain
 \[4\ceil{\log_pq}+2\leq m,\]
 which implies
 \[q^4<p^{4\ceil{\log_pq}}\leq p^{m-2}.\]
 Thus, if $p^{m-2}<q^4$ holds, then $\Fmax(N)$ must be necessarily empty, or in other words, $N\notin\Spq$, which by definition means that Fuglede's conjecture holds for cyclic groups of order
 $p^mq^n$, where $p^{m-2}<q^4$.

 \bigskip\bigskip

  \bibliographystyle{amsplain}

\providecommand{\bysame}{\leavevmode\hbox to3em{\hrulefill}\thinspace}
\providecommand{\MR}{\relax\ifhmode\unskip\space\fi MR }
\providecommand{\MRhref}[2]{%
  \href{http://www.ams.org/mathscinet-getitem?mr=#1}{#2}
}
\providecommand{\href}[2]{#2}

\end{document}